\documentclass[11pt]{amsart}
\usepackage{graphicx}
\usepackage{amsmath,amsfonts,amssymb,mathtools,amsthm}
\usepackage[dvipsnames]{xcolor}
\usepackage{soul}
\usepackage{csquotes}
\usepackage{hyperref}
\usepackage{bm}
\usepackage{enumitem}
\usepackage{microtype}
\usepackage[top=2.5cm, bottom=3cm, left=2.5cm, right=2.5cm]{geometry}
\usepackage{dirtytalk} 

\newcommand{\R}{\mathbb{R}}

\newcommand{\Z}{\mathbb{Z}}	

\DeclareMathOperator{\id}{id}
\renewcommand{\phi}{\varphi}
\newcommand{\aut}{\mathrm{Aut}(X,\mu)}
\newcommand{\Aut}{\mathrm{Aut}}

\DeclarePairedDelimiter{\abs}{|}{|}

\usepackage[
backend=biber,
style=numeric,
sortcites=true,
maxbibnames=5,
maxcitenames=5,
maxalphanames=5,
sorting=nyt,
giveninits=true,
doi=false,
url=false,
isbn=false
]{biblatex}
\DeclareFieldFormat*{title}{#1}
\DeclareNameAlias{default}{given-family}

\renewbibmacro{in:}{}
\theoremstyle{plain}
\newtheorem*{theorem*}{Theorem}
\newtheorem{theorem}{Theorem}[section]
\newtheorem{corollary}[theorem]{Corollary}
\newtheorem{theoremletter}{Theorem}

\newtheorem{corollaryletter}[theoremletter]{Corollary}

\newtheorem{lemma}[theorem]{Lemma}

\theoremstyle{definition}
\newtheorem{definition}[theorem]{Definition}

\newtheorem{remark}[theorem]{Remark}

\newtheorem*{notation}{Notation}

\numberwithin{equation}{section}

\title{Quantitative orbit equivalence for rank-one systems}

\author{Corentin Correia}
\address{Corentin Correia, 
Department of Mathematics and Statistics, University of Ottawa, 150 Louis-Pasteur Pvt, Ottawa, ON, K1N 9A7, Canada}
\email{corentin.correia@math.cnrs.fr}

\author{Spyridon Petrakos}
\address{Spyridon Petrakos, 
Department of Mathematical Sciences, Chalmers University of Technology and University of Gothenburg, SE-412 96 Gothenburg, Sweden}
\email{petrakos@chalmers.se}

\date{\today}

\begin{document}

\begin{abstract}
    We prove that any two rank-one systems are sub-$L^1$ orbit equivalent, fully clarifying quantitative orbit equivalence for a generic class of transformations. As a corollary, we establish unconditional optimality of Belinskaya's theorem.
\end{abstract}

\maketitle

\section{Introduction}

Quantitative orbit equivalence as a framework is, in principle, a bridge between orbit equivalence and conjugacy of free ergodic probability-measure-preserving (p.m.p.) actions. Our interest lies in the traditional setting of single transformations on a Lebesgue space, wherein such bridges are particularly important since one end of the spectrum is known to be trivial~\cite{dye59} and the other to be intractable~\cite{forRudWei11}.

An orbit equivalence between two free ergodic p.m.p.~transformations $T\in\Aut(X,\mu)$ and $S\in\Aut(Y,\nu)$ is a measure isomorphism $\varphi\colon X\to Y$ for which there exist maps $c_T\colon X\to\Z$ and $c_S\colon Y\to\Z$, called orbit cocycles, with
\[
\varphi(Tx)=S^{c_T(x)}\varphi(x)\quad\text{and}\quad\varphi^{-1}(Sy)=T^{c_S(y)}\varphi^{-1}(y)
\]
almost surely. Quantitativeness is instantiated through integrability conditions on the orbit cocycles. 
\begin{definition}
    Given an increasing map $\omega\colon\R_+\to\R_+$, two free ergodic transformations $T\in\Aut(X,\mu)$ and $S\in\Aut(Y,\nu)$ are \emph{$\omega$-integrably orbit equivalent} if an orbit equivalence $\varphi$ between them can be chosen so that $c_T$ and $c_S$ satisfy
\[
\int_X\omega(\abs{c_T(x)})d\mu(x)<\infty\quad\text{and}\quad\int_Y\omega(\abs{c_S(y)})d\nu(y)<\infty.
\]
    They are \emph{$L^p$ orbit equivalent}, where $0<p<\infty$, if they are $\omega$-integrably orbit equivalent for $\omega(r)=r^p$, and \emph{sub-$L^p$ orbit equivalent} if they are $\omega$-integrably orbit equivalent for all $\omega$ with
    \[
    \lim_{r\to\infty}\frac{\omega(r)}{r^p}=0.
    \]
\end{definition}

While the above definition is rather modern~\cite{delKoiLeMTes22}, the conceptual idea of controlling the orbit cocycles can be traced back to Rudolph's theory of restricted orbit equivalence~\cite{rudolphRestOrb85}. In fact, the first result in this context is even older, dating back to 1968.
\begin{theorem*}[Belinskaya~\cite{belinskaya68}]\label{thm:Belinskaya}
    If two ergodic transformations $T,S\in\Aut(X,\mu)$ are $L^1$ orbit equivalent then they are flip-conjugate.
\end{theorem*}

Another important layer of rigidity was obtained much more recently by Kerr and Li~\cite{kerrEntropyVirtualAbelianness2024}. They showed that entropy is preserved by Shannon orbit equivalence, a notion introduced by them in~\cite{kerLiSparseCon21}. The latter is implied by $\log$-integrable orbit equivalence~\cite[Theorem~3.16]{carderiBelinskayaTheoremOptimal2023}, tying it back our framework. They also exhibit an instance of non-rigidity by showing that every odometer is Shannon orbit equivalent to the universal one, a result that was subsequently improved to sub-$L^{1/3}$ orbit equivalence by the first named author~\cite{correiaRankoneSystemsFlexible2025}.

The aforementioned improvement also expanded the domain to a wider subclass of rank-one systems, but it remained one-directional. Most instances of quantitative equivalence (including Shannon) are not known to be transitive, and thus equivalence of each individual system to the universal odometer does not imply equivalence between any two of them. To deal with this pathology, Naryshkin and the second named author developed a different approach and used it to prove that any two odometers are sub-$L^1$ orbit equivalent---a bound that is optimal by Belinskaya's theorem~\cite{narPetOdom26}.

The main goal of the present work is to establish the same optimal bound for all rank-one systems.

\begin{theoremletter}\label{TheoremA}
    Any two rank-one systems are sub-$L^1$ orbit equivalent.
\end{theoremletter}

It is worth noting that while the class of rank-one systems only contains transformations of entropy zero---constant entropy is necessary by Kerr and Li's result after all---it is generic in the space of p.m.p.~transformations on a Lebesgue space (see~\cite[Lemme~5.26]{leMaitreThesis} for a proof). Furthermore, the dynamical behaviour of its members can vary dramatically, a feature exploited to produce counterexamples in ergodic theory since the very conception of such systems. In particular, various dynamical properties or quantities that could be natural candidates for rigidity results are not preserved by sub-$L^1$ orbit equivalence (e.g.~slow entropy). See Section~\ref{sec:PrelR1} for more details.

The proof of the theorem relies on adapting the strategy pioneered in~\cite{narPetOdom26}. In the case of odometers, there is a natural (projective) limit structure in which factor maps are combinatorially homogeneous---in the language of cutting-and-stacking, there are no spacers---and choosing representatives is straightforward. The strategy was then to inductively construct an approximate intertwining while controlling the cocycles at each step, a process reduced by homogeneity to the careful construction of a single auxiliary map. In the case of arbitrary rank-one systems, however, one has to contend with the absence of such regularity. To that end, we utilise both the classical cutting-and-stacking and the $(C, F)$-construction introduced by del Junco~\cite{deljuncoSimpleMapNo1998} and extensively explored in the work Danilenko and Vieprik~\cite{danilenkoActionsFiniteRank2016,danilenkoClassificationRankoneActions2025,danilenkoExplicitRank1Constructions2023,danilenkoFunnyRankoneWeak2001,danilenkoInfiniteRankOne2004,danilenkoRankoneActionsTheir2019,danilenkoRankoneNonsingularActions2024}. The former is used to prune all the necessary parameters involved in choosing a suitable representative of a given system, while the latter provides the direct limit structure needed to access approximate intertwining. Having established our model, we then conduct a thorough analysis of potential smaller-scale behaviour at a given stage in order to construct local auxiliary maps that are glued together to produce one step of the intertwining. This process is delicate and technically formidable, contrasting the rather neat landscape of odometers.

As a corollary of our main theorem we obtain the following, which was known in the special case where $S$ is an odometer~\cite{corOdomutants25}.

\begin{theoremletter}\label{TheoremB}
    Let $S\in\aut$ be an ergodic system such that for every $n\geq 2$, $S^n$ is not ergodic. Let $\omega\colon\R_+\to\R_+$ be a sublinear map. Then there exists $T\in\aut$ such that $S$ and $T$ are $\omega$-integrably orbit equivalent but not flip-conjugate.
\end{theoremletter}

When combined with~\cite[Theorem~1.3]{carderiBelinskayaTheoremOptimal2023}, the above implies unconditional optimality of Belinskaya's theorem.

\begin{corollaryletter}
    For any ergodic $S\in\aut$ and any sublinear map $\omega\colon\R_+\to\R_+$, there exists $T\in\aut$ such that $S$ and $T$ are $\omega$-integrably orbit equivalent but not flip-conjugate.
\end{corollaryletter}

The paper is structured as follows. Section~\ref{sec:PrelR1} provides an overview of rank-one systems. Section~\ref{sec:ProofThA} contains all the technical work, culminating in the proof of Theorem~\ref{TheoremA}. Section~\ref{sec:Belinskaya} is dedicated to the proof of Theorem~\ref{TheoremB}.

\textbf{Acknowledgements.} The first named author was partially supported by the Fields Institute for Research in Mathematical Sciences. The second named author was funded by the Knut and Alice Wallenberg Foundation through a postdoc grant.

\section{Preliminaries on rank-one systems}\label{sec:PrelR1}

\textbf{Rank-one systems defined via the cutting-and-stacking method.} Rank-one systems can be viewed through various (mostly) equivalent lenses. We will restrict to only two of them. We refer the interested reader to Ferenczi's survey~\cite{ferencziSystemsFiniteRank1997} for an illuminating Cook's tour of finite-rank systems. We first introduce the cutting-and-stacking picture below.

\begin{definition}\label{defr1}
		A transformation $T\in\aut$ is of \textbf{rank one} if there exist
		\begin{enumerate}
			\item a sequence $(q_n,(\sigma_{n,0},\ldots,\sigma_{n,q_n}))_{n\geq 0}$ of positive integer parameters satisfying
			\begin{equation}
				\displaystyle\sum_{n=0}^{+\infty}{\frac{\sigma_n}{k_{n+1}}}<+\infty,
			\end{equation}
			where $\sigma_n=\sigma_{0,n}+\sigma_{1,n}+\dots+\sigma_{q_n,n}$ and $(k_n)_{n\geq 0}$ is inductively defined by $k_0=1$ and $k_{n+1}=q_nk_n+\sigma_n$;
			\item measurable subsets of $X$, denoted by $B_n$ for every $n\geq 0$, $B_{n,i}$ for every $n\geq 0$ and $0\leq i\leq q_n-1$, and $\Sigma_{n,i,j}$ for every $n\geq 0$, $0\leq i\leq q_n$ and $1\leq j\leq \sigma_{n,i}$ (if $\sigma_{n,i}=0$, then the corresponding collection is empty) such that for all $n\geq 0$
			\begin{enumerate}
				\item $B_n,\ldots ,T^{k_n-1}(B_n)$ are pairwise disjoint;
				\item $(B_{n,0}, B_{n,1}, \ldots, B_{n,q_n-1})$ is a partition of $B_n$;
				\item \[T^{k_n}(B_{n,i})=
				\begin{cases}
					\Sigma_{n,i+1,1}&\text{ if }\sigma_{n,i}>0\\
					B_{n,i+1}&\text{  if }\sigma_{n,i}=0\text{  and }i<q_n-1
				\end{cases};\]
				\item if $\sigma_{n,i}>0$, then \[T(\Sigma_{n,i,j})=
				\begin{cases}
					\Sigma_{n,i,j+1}&\text{ if }j<\sigma_{n,i}\\
					B_{n,i}&\text{ if }j=\sigma_{n,i}\text{ and }i\leq q_n-1
				\end{cases};\]
				\item \[B_{n+1}=
				\begin{cases}
					\Sigma_{n,0,1}&\text{  if }\sigma_{n,0}>0\\
					B_{n,0}&\text{  if }\sigma_{n,0}=0
				\end{cases};\]
                \item the Rokhlin towers $\mathcal{R}_n\coloneq (T^k(B_n))_{0\leq k\leq k_n-1}$
		generate the $\sigma$-algebra $\mathcal{A}$ of measurable subsets.
			\end{enumerate}
		\end{enumerate}
        The tower $\mathcal{R}_n$ has height $k_n$. Note that $\mathcal{R}_0$ is the tower with only one level $B_0$. The sets $\Sigma_{n,i,j}$ are called \textbf{spacers} and $\sigma_n$ is the number of new spacers at step $n$. The integers $q_n$ and $\sigma_{n,i}$ are called the \textbf{cutting} and \textbf{spacing} parameters, respectively.
	\end{definition}

For a given sequence $(q_n,(\sigma_{n,0},\ldots,\sigma_{n,q_n}))_{n\geq 0}$ with $(\sigma_n/k_{n+1})_{n\geq 0}$ summable, it is always possible to build a rank-one system with these as parameters. See, for instance,~\cite[Lemma~3.4]{correiaRankoneSystemsFlexible2025} and the discussion thereafter.

The hypothesis on the Rokhlin towers $\mathcal{R}_n$ guarantees that rank-one systems are fully determined by their cutting-and-stacking constructions. If $T$ admits such a construction with Rokhlin towers increasing to a sub-$\sigma$-algebra $\mathcal{B}$ of $\mathcal{A}$, then $T$, seen as an element of $\Aut(X,\mathcal{A},\mu)$, is not necessarily a rank-one system but is an extension of one.

On the other hand, two different families of cutting and spacing parameters do not necessarily define non-isomorphic systems. Indeed, in the construction of a rank-one system with parameters $q_n$ and $\sigma_{n,i}$, one can decide to only consider a subsequence $\mathcal{R}_{n_k}$ of Rokhlin towers. For example, the new cutting parameters will be $q_{n_k}q_{n_k+1}\ldots q_{n_{k+1}-1}$ for $k\geq 0$. More generally, using another model of rank-one systems called the $(C,F)$-construction, Danilenko and Vieprik have devised a sufficient and necessary condition on the parameters for two rank-one systems to be isomorphic. We will return to this point after exploring the breadth of this class of systems.

Rank-one systems form a class of ergodic systems that is generic in $\aut$ endowed with the weak topology. They all have zero entropy, but exhibit a wide range of slow entropy behaviour; see for example~\cite[Theorem~5.9]{banerjeeSlowEntropyCombinatorial2023}. The following examples demonstrate considerable flexibility also in terms of spectral and mixing properties.

\textbf{Examples of rank-one systems.} Fundamental examples of rank-one systems are the irrational rotations
	\[R_{\theta}\colon z\in\mathbb{T}\mapsto e^{2i\pi\theta}z\in\mathbb{T}\]
for every irrational number $\theta$, where $\mathbb{T}$ is the unit circle endowed with its Haar measure. These systems are not weakly mixing. Moreover they have discrete spectrum and the point spectrum of $R_{\theta}$ is $\{e^{2i\pi n\theta}\mid n\in\Z\}$, so $R_{\theta}$ and $R_{\theta'}$ are isomorphic if and only if $\theta=\theta'\bmod \Z$ or $\theta=-\theta'\bmod \Z$.
    
Also fundamental are the odometers. These are exactly the rank-one systems without spacers (i.e.~$\sigma_{n,i}=0$), in which case the Rokhlin towers are partitions of the space. Such a system is isomorphic to the adding machine $S$ on a space of the form $\prod_{n\geq 0}{\{0,1,\ldots ,q_n-1\}}$, namely the addition by $(1,0,0,0,\ldots)$ with carry over to the right, and it preserves the product of uniform probability measures on each finite set $\{0,1,\ldots,q_n-1\}$. Denote the cylinders of length $k$ by
	\[[x_0,\ldots ,x_{k-1}]_k\coloneq \Big \{y\in\prod_{n\geq 0}{\{0,1,\ldots ,q_n-1\}}\mid y_0=x_0,\ldots ,y_{k-1}=x_{k-1}\Big \}.\]
Then Definition~\ref{defr1} is satisfied with
	\[B_n=[\underbrace{0,\ldots ,0}_{n\text{ times}}]_{n},\ B_{n,i}=[\underbrace{0,\ldots ,0}_{n\text{ times}},i]_{n+1},\ k_n=q_0q_1\dots q_{n-1}.\]
	
In the class of odometers, the number of occurrences of every prime factor in the set $\{q_n\mid n\geq 0\}$ forms a total invariant of conjugacy. It is a consequence of the Halmos-von Neumann Theorem since odometers have discrete spectrum and their eigenvalues are given by these occurrences. In particular, odometers have eigenvalues non-equal to $1$ and are not weakly mixing. Moreover odometers and irrational rotations are not isomorphic. Notice that the Halmos-von Neumann Theorem implies that the conjugacy classes among ergodic systems with discrete spectrum coincide with the flip-conjugacy classes since the point spectrum of a system is a subgroup of $\mathbb{T}$. 
    
Chacon's map is the first example of a weakly mixing system which is not strongly mixing~\cite{chaconWeaklyMixingTransformations1969a} and was the starting point for the theory of rank-one systems. It is a rank-one transformation defined with cutting and spacing parameters $q_n=3$, $\sigma_{n,0}=\sigma_{n,1}=\sigma_{n,3}=0$, $\sigma_{n,2}=1$.

Finally, there exist strongly mixing rank-one systems (and therefore of continuous spectrum) with rather exotic dynamical properties. The first such example was constructed by Ornstein~\cite{ornRootProb} using a probabilistic approach (\enquote{random spacers}).

\textbf{The $(C,F)$-construction.} We now move on to the model for rank-one systems that we will use for our construction of orbit equivalence. The $(C,F)$-construction can be considered as a generalization of the adding machines describing odometers. It was introduced by del Junco~\cite{deljuncoSimpleMapNo1998} and extensively used by Danilenko and Vieprik for explicit constructions of rank-one systems with prescribed properties, or for generalizations of rank-one systems as non-singular systems, or even as group actions (see e.g.~\cite{danilenkoFunnyRankoneWeak2001,danilenkoInfiniteRankOne2004,danilenkoActionsFiniteRank2016,danilenkoRankoneActionsTheir2019,danilenkoExplicitRank1Constructions2023,danilenkoRankoneNonsingularActions2024,danilenkoClassificationRankoneActions2025}).

A $(C,F)$-construction is encoded by sequences $(F_n)_{n\geq 0}$ and $(C_n)_{n\geq 0}$ of finite subsets of $\Z$ such that for every $n\geq 0$:
\begin{itemize}
    \item $F_0=\{0\}$;
    \item $|C_n|>1$;
    \item $F_n+C_n\subset F_{n+1}$;
    \item if $c,c'$ are different points of $C_n$ then $(F_n+c)\cap (F_n+c')=\emptyset$.
\end{itemize}
Throughout this paper, $F_n$ will be an interval of the form $\{0,1,\ldots,k_n-1\}$. The transformation provided by a $(C,F)$-construction as above (and which turns out to be isomorphic to a rank-one system) is formally defined as follows. We first set
$$X_n=F_n\times C_n\times C_{n+1}\times\dots$$
for every $n\geq 0$, and let $X$ be the inductive limit of these sets with respect to the embeddings
\[(f_n,c_n,c_{n+1},\dots)\in X_n\mapsto (f_n+c_n,c_{n+1},\dots)\in X_{n+1}.\]
We denote by $\iota^{(n)}\colon X_n\to X$ the canonical embedding. Given a subset $A$ of $F_n$, the cylinder over $A$ is the subset of $X$ defined by
$$[A]_n\coloneq\iota^{(n)}\left(\{(f_n,c_n,c_{n+1},\dots)\in X_n\mid f_n\in A\}\right),$$
and we simply write $[f]_n$ when $A$ is the singleton $\{f\}$.

The set $X$ is a locally compact space which can be endowed with the unique (up to multiplicative constant) $\sigma$-finite Borel measure $\mu$ satisfying $\mu([f]_n)=\mu([f']_n)$ for every $f,f'\in F_n$ and $n\geq 0$. The measure $\mu$ is finite if and only if the series
$$\sum_{n=1}^{+\infty}{\frac{|F_{n+1}|-|F_n|\cdot|C_n|}{|F_{n+1}|}}$$
converges. We assume that the latter holds and $\mu$ is chosen to be a probability measure.

We finally define $T\in\aut$ as follows. For almost every $x\in X$, there exists $n\geq 0$ such that $x$ can be written as $\iota^{(n)}(f_n,c_n,c_{n+1},\dots)$ with $f_n\in F_n$ such that $f_n+1\in F_n$. We set
$$Tx\coloneq\iota^{(n)}(f_n+1,c_n,c_{n+1},\dots).$$

In fact, the elements of $F_n$ encode the levels of the tower $\mathcal{R}_n$ in a cutting-and-stacking construction of a rank-one system, and $C_n$ provides a description of the levels of $\mathcal{R}_{n+1}$ that are included in a level of $\mathcal{R}_n$, in such a way that the cutting parameter corresponds to $q_n\coloneq |C_n|$ and, denoting $C_n=\{c_{n,0}<c_{n,1}<\dots<c_{n,q_n-1}\}$, the spacing parameters are exactly
\begin{itemize}
    \item $\sigma_{n,0}=c_{n,0}$;
    \item $\sigma_{n,i}=c_{n,i}-(c_{n,i-1}+k_n)$ for every $i\in\{1,2,\dots,q_n-1\}$;
    \item $\sigma_{n,q_n}=k_{n+1}-(c_{n,q_n-1}+k_n)$.
\end{itemize}
Moreover, the set $X_n$ plays the role of the subset covered by the tower $\mathcal{R}_n$.

The transformation $T$ provided by a $(C,F)$-construction is isomorphic to the rank-one system described by the corresponding parameters $k_n,q_n,\sigma_{n,i}$. We refer the reader to~\cite[Theorem~1.13]{danilenkoRankoneNonsingularActions2024} for a detailed proof. Finally, the convergence of the aforementioned series is equivalent to the summability of $(\sigma_n/k_{n+1})_{n\geq0}$.

We end this section with the classification up to isomorphism of rank-one systems, due to Danilenko and Vieprik. The general statement deals with rank-one group actions preserving a $\sigma$-finite measure. We state it in the particular case of p.m.p.~$\Z$-actions with subsets $F_n$ chosen as intervals, corresponding to the classical notion of rank-one transformations.

\begin{theorem}[{\cite[Theorem~A]{danilenkoClassificationRankoneActions2025}}]\label{(C,F)-iso}
    Let $T$ and $T'$ be two probability measure-preserving systems, with $(C,F)$-constructions described by sequences $(C_n,F_n)_{n\geq 0}$ and $(C'_n,F'_n)_{n\geq 0}$ respectively. Then $T$ and $\tilde T$ are isomorphic if and only if there exist increasing sequences $(k_n)_{n\geq 0}$ and $(l_n)_{n\geq 0}$ of non-negative integers, with $k_0=l_0=0$, and subsets $J_n\subseteq F_{k_n}$, $J'_n\subseteq F'_{l_n}$ such that
    \begin{enumerate}
        \item $F_{k_n}+J'_n\subseteq F'_{l_n}$ and
        \item the mapping $(f,f')\in F_{k_n}\times J'_n\mapsto f+f'\in F'_{l_n}$ is one-to-one, for all $n\geq 0$;
        \item the series
        \[\sum_{n\geq 0}{\frac{\big|(J'_n +J_{n+1})\ \Delta\  (C_{k_n}+\dots +C_{k_{n+1}-1})\big|}{|C_{k_n}|\dots |C_{k_{n+1}-1}|}}\]
        converges;
        \item $F'_{l_n}+J_{n+1}\subseteq F_{k_{n+1}}$ and
        \item the mapping $(f',f)\in F'_{l_n}\times J_{n+1}\mapsto f'+f\in F_{k_{n+1}}$ is one-to-one, for all $n\geq 0$;
        \item the series 
        \[\sum_{n\geq 0}{\frac{\big| (J_{n+1} +J'_{n+1})\ \Delta\  (C'_{l_n}+\dots +C'_{l_{n+1}-1})\big|}{|C'_{l_n}|\dots |C'_{l_{n+1}-1}|}}\]
        converges.
    \end{enumerate}
\end{theorem}

We warn the reader that our notational conventions differ slightly from the ones of Danilenko and Vieprik---the sequence $(C_n)$ starts from index $0$, so that we have $F_n+C_n\subset F_{n+1}$ instead of $F_n+C_{n+1}\subset F_{n+1}$.

\section{Proof of Theorem~\ref{TheoremA}}\label{sec:ProofThA}

\subsection{Preliminary work on the cutting and spacing parameters}

\begin{lemma}\label{lem:sublinear}
    Let $\omega\colon\R_+\to\R_+$ be a non-decreasing sublinear map. Then there exists a non-decreasing map $\zeta\colon\R_+\to\R_+$ satisfying the following properties:
    \begin{itemize}
        \item $\zeta(x)\geq 1$ for every $x\in\R_+$;
        \item $\zeta$ is sublinear;
        \item $\lim_{x\to\infty}\zeta(x)=\infty$;
        \item $\lim_{x\to\infty}(x^{-1}\cdot\omega(x\zeta(x))\cdot\zeta(x))=0$.
    \end{itemize}
    In particular, we also have 
    \[
    \lim_{x\to\infty}\frac{\omega(x)\cdot\zeta(x)}{x}=0.\]
\end{lemma}

\begin{proof}
    By sublinearity of $\omega$, we can inductively build an increasing sequence $x_1<x_2<\dots$ of real numbers such that
    \[
        \frac{\omega (x (n+1))\cdot (n+1)}{x}\leq \frac{1}{n+1}.
    \]
    for every $n\geq 1$ and $x\geq x_n$. We then set $\zeta(x_n)= n$ and $\zeta(x)=1$ for every $x\in [0,x_1]$, and we extend $\zeta$ linearly on $\R_+$. If we choose $(x_n)_{n\geq 1}$ to grow fast enough to $\infty$, then $\zeta$ is sublinear. It is not hard to prove that
    \[
        \lim_{x\to\infty}\frac{\omega(x\zeta(x))\cdot\zeta(x)}{x}=0,
    \]
    using the fact that $\omega$ is non-decreasing.
\end{proof}

\begin{lemma}\label{lem:ProperlyChosenCuttingAndStacking1}
    Let $T\in\aut$ be a rank-one system, and let $\zeta\colon\R_+\to\R_+$ be a sublinear map such that $\lim_{x\to\infty}\zeta(x)=\infty$. We consider a cutting-and-stacking construction of $T$ with parameters $k'_n,q'_n,\sigma'_{n,i}$ satisfying
    \begin{itemize}
        \item $\sum_{n\geq 1}(\sigma'_n/k'_{n+1})\leq 1/2$;
        \item $\zeta(k'_n)>4$ and $k'_n/\zeta(k'_n)>1$ for every $n\geq 1$;
        \item the sequences $(1/\zeta(k'_n))_n$ and $(1/q_n)_n$ are summable.
    \end{itemize}
    Then there exists another cutting-and-stacking construction of $T$ with parameters $k_n,q_n,\sigma_{n,i}$ satisfying the following properties for every $n\geq 1$:
    \begin{enumerate}
        \item $k_{n}= k'_n-\left\lceil k'_{n}/\zeta(k'_{n})\right\rceil$;
        \item $\sigma_{n,i}\geq k'_{n}/\zeta(k'_{n})$ for every $i\in\{1,\dots,q_n-1\}$.
    \end{enumerate}
\end{lemma}

\begin{proof}
    We consider the sets $C'_n$ and $F'_n$ in the $(C,F)$-construction associated to the parameters $k'_n,q'_n,\sigma'_{n,i}$. Setting $N_n=\lceil4q'_n/\zeta(k'_{n+1})\rceil$, we denote by $C_n$ the set obtained from $C'_n$ by removing the last $N_n$ elements ($C'_n$ has more than $N_n$ elements since $\zeta(k'_n)>4$). We also define $F_n$ for $n>0$ ($F_0$ is always $\{0\}$) by removing the last $\lceil k'_n/\zeta(k'_{n})\rceil$ elements from $F'_n$. We have
    \[
    N_nk'_n\geq \frac{4q'_nk'_n}{\zeta(k'_{n+1})}=\frac{4(k'_{n+1}-\sigma'_n)}{\zeta(k'_{n+1})}=4\left (1-\frac{\sigma'_n}{k'_{n+1}}\right)\frac{k'_{n+1}}{\zeta(k'_{n+1})}\geq 2\frac{k'_{n+1}}{\zeta(k'_{n+1})}\geq \left\lceil\frac{k'_{n+1}}{\zeta(k'_{n+1})}\right\rceil,
    \]
    which implies $F'_n+C_n\subset F_{n+1}$ (the last inequality comes from the fact that $k'_{n+1}/\zeta(k'_{n+1})>1$). In particular we have $F_n+C_n\subset F_{n+1}$, so the sets $C_n$ and $F_n$ provide a well-defined $(C,F)$-construction for which we denote by $k_n,q_n,\sigma_{n,i}$ the associated parameters.
    
    We immediately get $k_n=k'_n-\lceil k'_n/\zeta(k'_n)\rceil$. Given $i\in\{1,\dots,q_n\}$, $\sigma_{n,i}$ is the gap between the last element of $F_n+c_i$ and the first element of $F_n+c_{i+1}$, where $c_1<c_2<\dots <c_{q_n}$ is an enumeration of the elements of $C_n$. From the fact that this gap contains $(F'_n\setminus F_n)+c_i$, we deduce that $\sigma_{n,i}\geq k'_n/\zeta(k'_n)$.
    
    Using our assumptions, one easily checks that the assumptions of Theorem~\ref{(C,F)-iso} (with matching notation) are satisfied for
    \[
        k_n=l_n=n,\quad J'_n=\{0\},\quad J_{n+1}=C_{n}
    \]
    when $n>0$, and
    \[
        k_0=l_0=0,\quad
        J_0=J'_0=J_1=\{0\}.
    \]
    This implies that the two $(C,F)$-constructions give isomorphic rank-one systems, so $T$ admits a cutting-and-stacking construction with parameters $k_n,q_n,\sigma_{n,i}$.
\end{proof}

\begin{remark}\label{rem:subsequence}
    The prerequisites of Lemma~\ref{lem:ProperlyChosenCuttingAndStacking1} can always be assumed by skipping steps in a given construction. Indeed, when considering a subsequence $(\mathcal{R}_{n_j})$ of Rokhlin towers, the new height and cutting sequences are $(K'_j\coloneq k'_{n_j})_j$ and $(Q'_j\coloneq q'_{n_j}q'_{n_j+1}\dots q'_{n_{j+1}-1})_j$, and the new $j$-th spacing parameter is equal to
    \[\Sigma'_{j}\coloneq\sum_{n_j\leq n<n_{j+1}}{\sigma'_n\prod_{n< k< n_{j+1}}{q'_k}}.\]
    Therefore, we get
    \[\frac{\Sigma'_j}{K'_{j+1}}\leq\sum_{n_j\leq n<n_{j+1}}{\frac{\sigma_n}{k'_{n+1}}}.\]
    From this we also deduce that the assumptions of the lemma are stable under skipping steps. This is the reason why we keep the strong assumption on $\sum{(\sigma'_n/k'_{n+1})}$, whereas we only use the fact that $\sigma'_n/k'_{n+1}\leq 1/2$ for every $n\geq 1$.
\end{remark}

\subsection{Overview of the construction}\label{sec:overview}

Let $T^e\colon X^e\to X^e$ and $T^o\colon X^o\to X^o$ be two rank-one systems, with parameters $k'_{2n},q'_{2n},\sigma'_{2n,i}$ and $k'_{2n-1},q'_{2n-1},\sigma'_{2n-1,i}$. Note that the sets and parameters are now indexed on the even (resp.~odd) integers, so moving from one step of the construction to the next equates to passing from index $2n$ to index $2n+2$ (resp.~$2n+1$ to $2n+3$).

Let $\omega\colon\R_+\to\R_+$ be a sublinear map and $\zeta\colon\R_+\to\R_+$ be a map as in Lemma~\ref{lem:sublinear}. Note that, using~\cite[Lemma~2.12]{carderiBelinskayaTheoremOptimal2023}, we can assume without loss of generality that $\omega$ is increasing. Skipping steps in the cutting-and-stacking construction, we assume that $\sum_{n\geq 1}{(\sigma'_n/k'_{n+1})}\leq 1/2$ as in Lemma~\ref{lem:ProperlyChosenCuttingAndStacking1}. We set $k_n\coloneq k'_n-\left\lfloor k'_n/\zeta(k'_n)\right\rfloor$. We can also assume without loss of generality (see Remark~\ref{rem:subsequence}) that:
\begin{itemize}
    \item $k_{n+2}>k_{n+1}>k_n$ and $k_{n+2}>k_{n+1}k_n$ for every $n\geq 0$, for the construction presented in the next section to be well-defined;
    \item all the assumptions of Lemma~\ref{lem:ProperlyChosenCuttingAndStacking1} hold;
    \item the series $\sum_n (k_{n-1}k_{n-2}/k_n)$ converges, as in Lemma~\ref{lem:ProperlyChosenCuttingAndStacking2} stated below;
    \item other summability conditions appearing later, only involving $(k_n)$ and $(k'_n)$, are satisfied (details will be provided in the final proof of Theorem~\ref{TheoremA}).
\end{itemize}

By Lemma~\ref{lem:ProperlyChosenCuttingAndStacking1}, $T^e$ (resp.~$T^o$) admits a cutting-and-stacking construction with parameters $(k_{2n})_{n\geq 0}$, $(q_{2n})_{n\geq 0}$, $(\sigma_{2n,i})_{n\geq 0}$ (resp.~$(k_{2n-1})_{n\geq 0}$, $(q_{2n-1})_{n\geq 0}$, $(\sigma_{2n-1,i})_{n\geq 0}$) satisfying $\sigma_{n,i}\geq k'_n/\zeta(k'_n)$ for every $n\geq 1$. By convention, $k_0=k_{-1}=1$.

We denote by $C_{2n},F_{2n}$ and $C_{2n-1},F_{2n-1}$ the sets involved in the $(C,F)$-constructions associated with $k_{2n},q_{2n},\sigma_{2n,i}$ and $k_{2n-1},q_{2n-1},\sigma_{2n-1,i}$, respectively. Recall that $F_n$ is the set $\{0,1,\dots,k_n-1\}$. For every $i\in\{0,\dots,d_n-1\}$, with $d_n\coloneq \lfloor k_n/(k_{n-1}k_{n-2})\rfloor$, we set $I_{n,i}\coloneq ik_{n-1}k_{n-2}+[k_{n-1}k_{n-2}]$ and $I_{n,\ast}=F_n\setminus(I_{n,0}\sqcup\dots\sqcup I_{n,d_n-1})$ of cardinality denoted by $r_n$ (note that $r_n<k_{n-1}k_{n-2}$).

\begin{figure}[ht]
		\centering
		\includegraphics[width=0.9\linewidth]{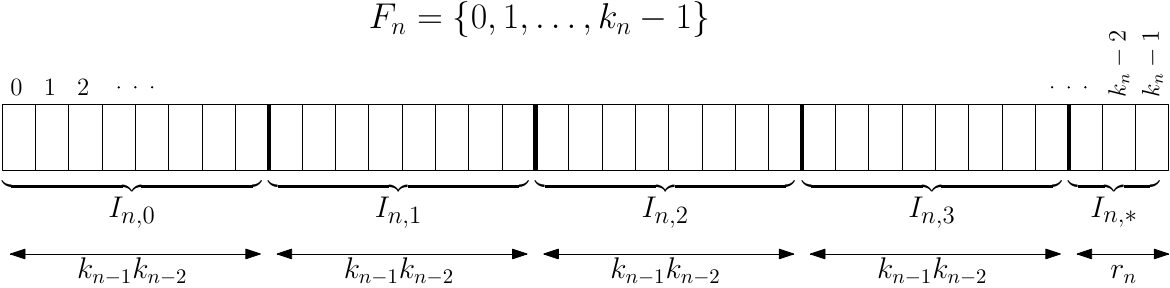}
		\caption{In this example, $k_n=35$, $k_{n-1}k_{n-2}=8$, $d_n=4$, $r_n=3$}
	\end{figure}

The following useful lemma asserts that we can remove the sets $F_n+c$, where $c\in C_n$, which are in the wrong place for our construction.

\begin{lemma}\label{lem:ProperlyChosenCuttingAndStacking2}
    Assume that $((k_{n-1}k_{n-2})/k_n)_n$ is summable. Then we can assume without loss of generality that for every $i\in [d_{n+1}]$, the set $F_{n-1}+C_{n-1}$ does not intersect the last $k_{n-1}r_{n}$ elements of $I_{n+1,i}$.
\end{lemma}

\begin{proof}
    If $F_{n-1}+C_{n-1}$ does not satisfy the property described in the statement of the lemma, then we consider the maximal subset $C'_{n-1}\subset C_{n-1}$ such that $F_{n-1}+C'_{n-1}$ does. To obtain $C'_{n-1}$ we have to remove at most $d_{n+1}(r_n+1)$ elements from $C_{n-1}$. We have
    \[
        d_{n+1}(r_n+2)\leq\frac{2k_{n+1}k_{n-2}}{k_n}=|C_{n-1}|\left (1+\frac{\sigma_{n-1}}{q_{n-1}k_{n-1}}\right)\frac{2k_{n-2}k_{n-1}}{k_n}.
    \]
    Since $\sigma_{n-1}/(q_{n-1}k_{n-1})\leq \sigma_{n-1}/(q_0\dots q_{n-1})$, we know that $(\sigma_{n-1}/(q_{n-1}k_{n-1}))_n$ is summable and therefore bounded above. Finally, $(k_{n-1}k_{n-2}/k_n)_n$ is also summable, so we deduce from~\cite[Theorem~A]{danilenkoClassificationRankoneActions2025} (as in the proof of Lemma~\ref{lem:ProperlyChosenCuttingAndStacking1}) that the two $(C,F)$-constructions yield isomorphic rank-one systems.
\end{proof}

For every $n\geq 1$ and $i\in\{0,1,\dots,d_{n+1}-1\}$, we will define a bijection
\[\psi_{n,i}\colon I_{n+1,i}\to [k_nk_{n-1}]\]
and a map $\varphi_n\colon F_{n+1}\to F_n$ defined by
\[ 
    \forall x\in F_{n+1},\ \varphi_n(x)=\begin{cases}
                    \psi_{n,i}(x)\bmod k_n&\text{ if }x\in I_{n+1,i}\\
                    x\bmod k_n&\text{ if }x\in I_{n+1,\ast}
                \end{cases}.
\]
With these maps $\varphi_n$, we will build an orbit equivalence between $T^e$ and $T^o$.

Notation introduced in this section will be used throughout. Furthermore, all parameters of rank-one systems are henceforth assumed to be as in Lemmas~\ref{lem:ProperlyChosenCuttingAndStacking1} and~\ref{lem:ProperlyChosenCuttingAndStacking2}.

\subsection{Definition of the maps \texorpdfstring{$\psi_{n,i}$}{}.}\label{sec:Construction}

For every $i\in [d_2]$, we define $\psi_{1,i}=\bmod k_1$ on $I_{2,i}$.

Assume that $\psi_{n-1,j}\colon I_{n,j}\to [k_{n-1} k_{n-2}]$ has been defined for every $j\in [d_n]$. Given $i\in [d_{n+1}]$, we now define $\psi_{n,i}\colon I_{n+1,i}\to [k_n k_{n-1}]$.

For every $\alpha\in [k_{n-1}]$, set $R_{n+1,i,\alpha}=(i+1)k_{n}k_{n-1} - (k_{n-1}-\alpha)r_n+[r_n]$ and define $\psi_{n,i}$ on each $R_{n+1,i,\alpha}$ as an increasing bijection $R_{n+1,i,\alpha}\to \alpha k_n+I_{n,\ast}$. It remains to define $\psi_{n,i}$ on $I_{n+1,i}\setminus (R_{n+1,i,0}\sqcup\dots\sqcup R_{n+1,i,k_{n-1}-1})$. Note that the cardinality of the latter is $k_{n-1}d_n\cdot k_{n-1}k_{n-2}$. The goal is now to partition this set into $k_{n-1}d_n$ subsets $J_{n,i,0},J_{n,i,1},\ldots ,J_{n,i,k_{n-1}d_n-1}$ of cardinality $k_{n-1}k_{n-2}$ so that $\psi_{n,i}$ induces bijections $J_{n,i,\alpha d_{n}+\beta}\to \alpha k_n+I_{n,\beta}$ for every $\alpha\in [k_{n-1}]$ and $\beta\in [d_n]$. We build these subsets as disjoint unions of blocks of cardinality $k_{n-1}$. These blocks are of two types.
    \begin{itemize}
        \item\textbf{Blocks of the first type.} Let
        \[C_{n-1,i}\coloneq\{c\in C_{n-1}\mid F_{n-1}+c\subset I_{n+1,i}\setminus (R_{n+1,i,0}\sqcup\dots\sqcup R_{n+1,i,k_{n-1}})\},\]
        which is also equal to $\{c\in C_{n-1}\mid (F_{n-1}+c)\cap I_{n+1,i}\not=\emptyset\}$ by Lemma~\ref{lem:ProperlyChosenCuttingAndStacking2}. The blocks of the first type are exactly the sets $F_{n-1}+c$, for $c\in C_{n-1,i}$.
        
        \item\textbf{Blocks of the second type.} In $I_{n+1,i}\setminus (R_{n+1,i,0}\sqcup\dots\sqcup R_{n+1,i,k_{n-1}})$, we remove all blocks of the first type, namely $F_{n-1}+C_{n-1,i}$. We get a set whose cardinality is a multiple of $k_{n-1}$, say $Mk_{n-1}$ for some integer $M\geq 0$, which we enumerate by $x_0<\dots<x_{Mk_{n-1}-1}$. The blocks of the second type are thus defined as the sets
        \[\{x_{\ell k_{n-1}},x_{\ell k_{n-1}+1},\dots,x_{(\ell+1) k_{n-1}-1}\}\]
        for every $\ell\in [M]$.
    \end{itemize}
    We now enumerate the blocks (of both types) $B_{n-1,i,0},B_{n-1,i,1},\dots,B_{n-1,i,k_{n-2}k_{n-1}d_n-1}$ by requiring that $\min{B_{n-1,i,\ell}}<\min{B_{n-1,i,\ell+1}}$ for every $\ell\in [k_{n-2}k_{n-1}d_n-1]$. Then, for every $r\in [k_{n-1}d_n]$, we define
    \[J_{n,i,r}=\bigsqcup_{0\leq \ell\leq k_{n-2}-1}{B_{n-1,i,rk_{n-2}+\ell}}.\]

    \begin{notation}
    For the sequel, we will need the following. We denote by $2_{n-1,i}$ the set of integers $j\in [k_{n-2}k_{n-1}d_n]$ such that $B_{n-1,i,j}$ is of the second type. Given $j,j'\in [k_{n-2}k_{n-1}d_n]$ such that $B_{n-1,i,j}\in 2_{n-1,i}$ and $B_{n-1,i,j'}$ is of the first type, we say that $B_{n-1,i,j'}$ \textit{is inside} $B_{n-1,i,j}$ if $B_{n-1,i,j'}\subset\{\min{B_{n-1,i,j}},\dots,\max{B_{n-1,i,j}}\}$. We denote by $1_{n-1,i}$ the set of integers $j\in k_{n-2}k_{n-1}d_n$ such that $B_{n-1,i,j}$ is of the first type and not inside a block of the second type.
    \end{notation}
    
    Writing $r=\alpha d_n+\beta$ with $\alpha\in [k_{n-1}]$ and $\beta\in [d_n]$, we define $\psi_{n,i}$ on $J_{n,i,r}$ as follows. We first consider the unique bijection $\xi_{n,i,r}\colon J_{n,i,r}\to [k_{n-1}k_{n-2}]$ inducing an increasing map $B_{n-1,i,rk_{n-2}+\ell}\to \ell k_{n-1}+[k_{n-1}]$ for every $\ell\in [k_{n-2}]$. If $B_{n-1,i,rk_{n-2}+\ell}$ is of the first type, then $\psi_{n,i}\coloneqq\alpha k_n+\psi^{-1}_{n-1,\beta}\circ\xi_{n,i,r}$ on this block. More precisely, if $B_{n-1,i,rk_{n-2}+\ell}$ is of the form $F_{n-1}+c$ for some $c\in C_{n-1,i}$, then we have
        \[\psi_{n,i}=\alpha k_n + \psi_{n-1,\beta}^{-1}(\cdot -c+\ell k_{n-1})\]
    on this set. Then we complete $\psi_{n,i}$ to a bijection $J_{n,i,r}\to\alpha k_n+I_{n,\beta}$ by imposing that $\psi_{n,i}$ is increasing on the union of blocks of the second type. Doing this for every such $r$ concludes the definition of $\psi_{n,i}\colon I_{n+1,i}\to [k_nk_{n-1}]$ (see Figure~\ref{fig:construction}).

    \begin{figure}[p]\label{fig:construction}
		\centering
		\includegraphics[width=1\linewidth]{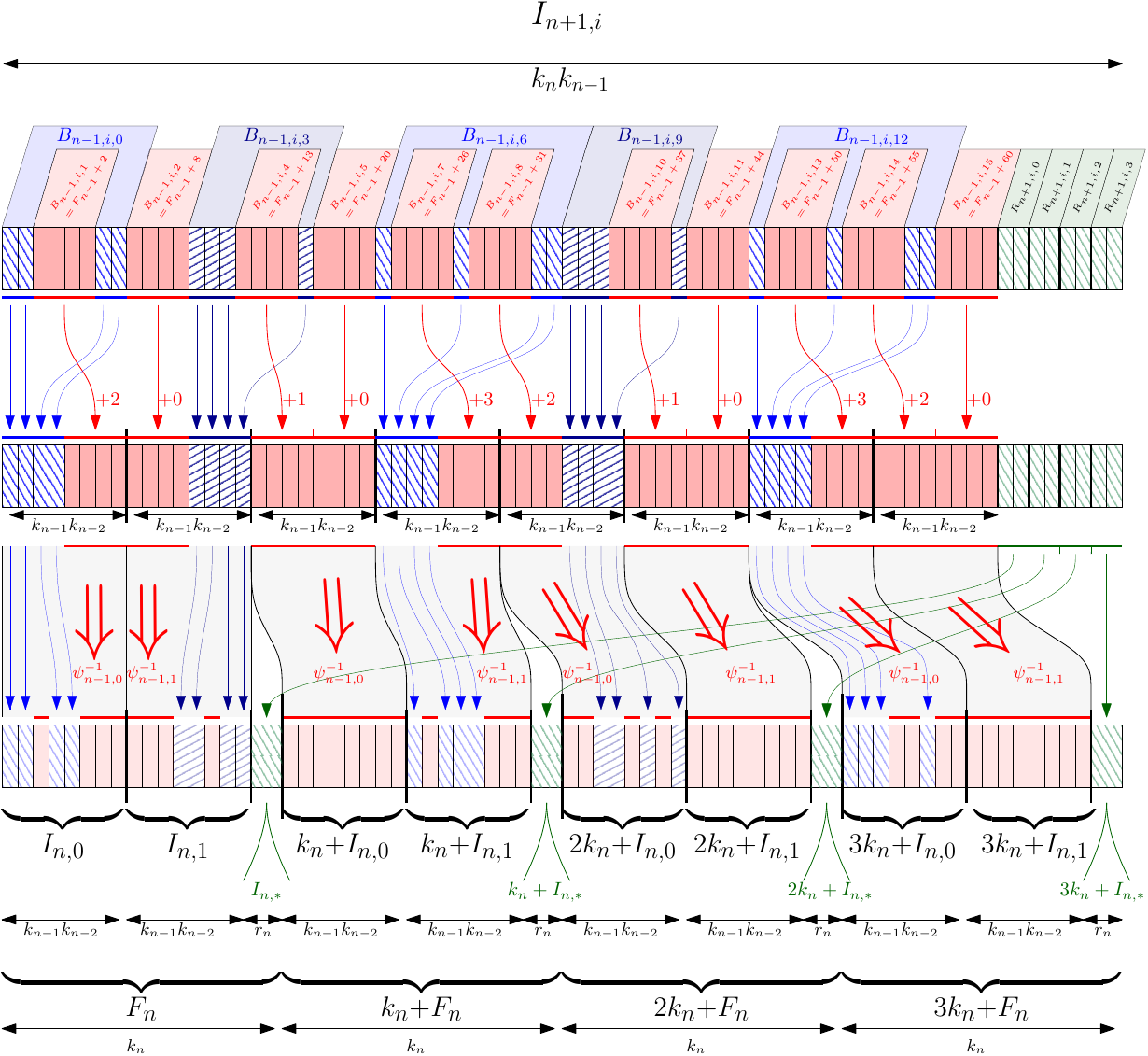}
		\caption{\small{Definition of $\psi_{n,i}\colon I_{n+1,i}\to [k_nk_{n-1}]$. In this example, $k_{n-2}=2$, $k_{n-1}=4$, $k_n=18$. We have $J_{n,i,0}=B_{n-1,i,0}\sqcup B_{n-1,i,1}$, $J_{n,i,1}=B_{n-1,i,2}\sqcup B_{n-1,i,3}$, etc. Moreover $2_{n-1,i}=\{0,3,6,9,12\}$, $1_{n-1,i}=\{2,5,11,15\}$ and the following blocks of first type are inside blocks of second type: $B_{n-1,i,1}$, $B_{n-1,i,4}$, $B_{n-1,i,7}$, $B_{n-1,i,8}$, $B_{n-1,i,10}$, $B_{n-1,i,13}$ and $B_{n-1,i,14}$.
        The last levels (shaded in green) are mapped to the copies of $I_{n,\ast}$ (green arrows). Outside these levels, the map $\psi_{n,i}$ is defined in two steps. We first rearrange the levels for the blocks to become intervals of size $k_{n-1}$ (the red ones are of first type, the ones shaded in blue are of second type). We then use the inverse maps $\psi_{n-1,\beta}^{-1}$ on the blocks of first type (red arrows) and complete using the unique increasing map (blue arrows).
        Note that the full description would be $F_{n-1}+2+ik_{n}k_{n-1}$ instead of $F_{n-1}+2$ (and similarly for the other blocks of first type), and $\alpha k_n+(\psi_{n-1,\beta}^{-1}\circ\bmod k_{n-1}k_{n-2})$ instead of $\psi_{n-1,\beta}^{-1}$ in the drawing. Finally, the maps $\xi_{n,i,r}$ correspond to the first step (rearrangement of the levels), composed with the $\bmod k_{n-1}k_{n-2}$ map we omit in the second step.}}
	\end{figure}

    \begin{remark}\label{rem:ToSumUp}
        To sum up:
        \begin{itemize}
            \item $\psi_{n,i}$ is a bijection from $I_{n+1,i}$ to $[k_nk_{n-1}]$.
            \item $I_{n+1,i}$ and $[k_nk_{n-1}]$ can be written as $[k_nk_{n-1}]=\bigsqcup_{\alpha\in [k_{n-1}]}{(\alpha k_n+[k_n])}$ and $I_{n+1,i}=\bigsqcup_{\alpha\in [k_{n-1}]}{I_{n+1,i,\alpha}}$, where
            \[I_{n+1,i,\alpha}\coloneq\left (\bigsqcup_{0\leq \beta\leq d_n-1}{J_{n,i,\alpha d_n+\beta}}\right )\sqcup R_{n+1,i,\alpha}\]
            for every $\alpha\in [k_{n-1}]$, and $\psi_{n,i}$ restricts to a bijection $I_{n+1,i,\alpha}\to\alpha k_n +[k_n]$.
            \item The definition of $I_{n+1,i,\alpha}$ provides itself a partition of the set, and $\alpha k_n +[k_n]$ can be written as
            \[\alpha k_n +[k_n]=\left (\bigsqcup_{0\leq\beta\leq d_n-1}{(\alpha k_n+I_{n,\beta})}\right )\sqcup (\alpha k_n+I_{n,\ast}).\]
            Then $\psi_{n,i}$ restricts to a bijection $R_{n+1,i,\alpha}\to\alpha k_n+I_{n,\ast}$ and, for every $\beta\in [d_n]$, a bijection $J_{n,i,\alpha d_n+\beta}\to\alpha k_n+I_{n,\beta}$.
            \item Given $\alpha\in [k_{n-1}]$, $\beta\in [d_n]$, and writing $r=\alpha d_n+\beta$,
            \[J_{n,i,r}=\left (\bigsqcup_{\substack{l\in [k_{n-2}]\\rk_{n-2}+\ell\in 2_{n-1,i}}}{B_{n-1,i,rk_{n-2}+\ell}}\right )\sqcup\bigsqcup_{\substack{l\in [k_{n-2}]\\rk_{n-2}+\ell\not\in 2_{n-1,i}}}{B_{n-1,i,rk_{n-2}+\ell}}.\]
            Using the definition, it is not hard to prove that
            \[\psi_{n,i}(B_{n-1,i,rk_{n-2}+\ell})=\alpha k_n+I_{n,\beta,\ell}\]
            for every $l\in [k_{n-2}]$ such that $rk_{n-2}+\ell\not\in 2_{n-1,i}$, and
            \[\psi_{n,i}\left (\bigsqcup_{\substack{l\in [k_{n-2}]\\rk_{n-2}+\ell\in 2_{n-1,i}}}{B_{n-1,i,rk_{n-2}+\ell}}\right )=\bigsqcup_{\substack{l\in [k_{n-2}]\\rk_{n-2}+\ell\in 2_{n-1,i}}}{(\alpha k_n+I_{n,\beta,\ell})}.\]
        \end{itemize}
    \end{remark}
    
\subsection{Properties of the blocks}\label{sec:PropBlock}

    Let $n\geq 2$, $i\in [d_n]$. We summarize here the properties of the blocks $B_{n-1,i,\ell}$, for $\ell\in [k_{n-2}k_{n-1}d_n]$. Note that $\sigma_{n-1,j}\geq k'_{n-1}/\zeta(k'_{n-1})$ for every $j\in\{0,\dots,q_{n-1}\}$.
    \begin{enumerate}[label=(\Roman*)]
        
        \item Every block has size $k_{n-1}$.
        
        \item Blocks of the first type are intervals of the form $F_{n-1}+c$ with $c\in C_{n-1}$.
        
        \item\label{item:3} Given $\ell\in [k_{n-2}k_{n-1}d_n]\cap 2_{n-1,i}$, if $F_{n-1}+c_1,\ldots ,F_{n-1}+c_N$ denote the blocks of the first type inside $B_{n-1,i,\ell}$, with $c_1<\ldots <c_N\in C_{n-1}$ (if $B_{n-1,i,\ell}$ is an interval, then $N=0$), then the following hold.
        \begin{itemize}
        
            \item For every $i\in\{1,\ldots,N\}$, $B_{n-1,i,\ell+i}=F_{n-1}+c_i$.
            
            \item $N\leq \zeta(k'_{n-1})+1$. Indeed, $B_{n-1,i,\ell}$ has cardinality $k_{n-1}< k'_{n-1}$, and between any two consecutive blocks $F_{n-1}+c_j$ and $F_{n-1}+c_{j+1}$, $j\in\{1,\ldots,N-1\}$, there are at least $\frac{k'_{n-1}}{\zeta(k'_{n-1})}$ spacers (lying in $B_{n-1,i,\ell}$).
            
            \item Between any two consecutive blocks $F_{n-1}+c_j$ and $F_{n-1}+c_{j+1}$, $j\in\{1,\ldots,N-1\}$ there are less than $k_{n-1}$ elements, since these are elements of $B_{n-1,i,\ell}$ and there are at least two more spacers (the edges).

            \item $\ell+N+1\in 1_{n-1,i}\cup 2_{n-1,i}$, and $\min{B_{n-1,i,\ell+N+1}}=1+\max{B_{n-1,i,\ell}}$.
            
        \end{itemize}
        
        \item We deduce from the above description that, given $\ell,\ell'\in [k_{n-2}k_{n-1}d_n]$, if there exists $x\in B_{n-1,i,\ell}$ such that $x+1\in B_{n-1,i,\ell'}$, then $|\ell'-\ell|\leq \zeta(k'_{n-1}) +2$.
        
        \item\label{item:5} Given $r,r'\in [k_{n-1}d_n]$, if there exists $x\in J_{n,i,r}$ such that $x+1\in J_{n,i,r'}$, then 
        \[|r'-r|\leq 1+\frac{\zeta(k'_{n-1}) +2}{k_{n-2}}.\]
        Indeed, applying the previous point to $\ell,\ell'\in [k_{n-2}]$ such that $x\in B_{n-1,i,rk_{n-2}+\ell}$ and $x+1\in B_{n-1,i,r'k_{n-2}+\ell'}$ yields
        \[(|r-r'|-1)k_{n-2}\leq |rk_{n-2}+\ell -(r'k_{n-2}+\ell')|\leq  \zeta(k'_{n-1}) +2.\]
        This gives the desired bound.
    
        \item\label{item:6} Let $E=\bigsqcup_{u\leq\ell\leq v}{B_{n-1,i,\ell}}$, with $u,v\in [k_{n-2}k_{n-1}d_n]$ such that $u\leq v$. Utilizing the above, $E$ can be described as follows. We write $\{u,u+1,\dots,v\}\cap (1_{n-1,i}\sqcup 2_{n-1,i})=\{\ell_{1}<\ldots <\ell_{N}\}$ and denote by $\ell_0$ the greatest integer in $\{0,\dots,u-1\}\cap 2_{n-1,i}$, if it exists. We set $\ell_{\ast}\coloneq\ell_{N}-1$ if $\ell_N\in 2_{n-1,i}$ and some blocks inside $B_{n-1,\ell_N}$ are not in $E$, otherwise we set $\ell_{\ast}\coloneq v$. Then $E$ can be written as
        \[E=E_1\sqcup E_2\sqcup E_3,\]
        where:
        \begin{itemize}
            
            \item $E_1$ is empty if $\ell_0$ does not exist, otherwise it is the (possibly empty) union of the blocks of the first type $B_{n-1,i,u},\dots,B_{n-1,i,\ell_1-1}$ (these are the last blocks inside $B_{n-1,i,\ell_0}$). If it is not empty, then $E_1$ has at most $\zeta(k'_{n-1})+1$ connected components, of size $k_{n-1}$, and the gaps between these components are part of a single block of the second type, so the union of the gaps has cardinality less than $k_{n-1}$;
            
            \item $E_3$ is empty when $\ell_{\ast}$ equals $v$, otherwise it is $B_{n-1,i,\ell_N}\sqcup B_{n-1,i,\ell_N+1}\sqcup\dots\sqcup B_{n-1,i,v}$, namely $B_{n-1,i,\ell_N}$ along with an initial segment of blocks inside it;
            
            \item $E_2$ is the union of the blocks $B_{n-1,i,\ell}$, over $\ell\in\{\ell_1,\ell_1+1,\ldots,\ell_{\ast}\}$. $E_2$ is an interval, since for every $\ell\in\{\ell_1,\ell_1+1,\ldots,\ell_{\ast}\}\cap 2_{n-1,i}$, all the blocks inside $B_{n-1,i,\ell}$ are in $E$.
        
        \end{itemize}
        From the above estimates, we deduce that $E$ admits at most $2(\zeta(k'_{n-1})+1)+1$ connected components and
        \begin{align*}\max{E}-\min{E}&\leq (v-u+1)k_{n-1}+ k_{n-1}+(\zeta(k'_{n-1})+1)k_{n-1}\\
        &= (v-u+1)k_{n-1}+(\zeta(k'_{n-1})+2)k_{n-1}.\end{align*}

        \begin{figure}[ht]
		\centering
		\includegraphics[width=1\linewidth]{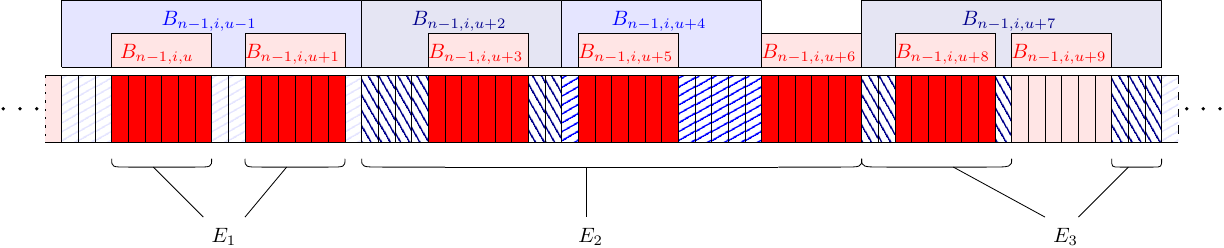}
		\caption{Illustration of $E=E_1\sqcup E_2\sqcup E_3$ with $v=u+8$.}
	\end{figure}
        
        \item\label{item:7} Let $r\leq r'\in [k_{n-1}d_n]$. Applying item~\ref{item:6} to $u=rk_{n-2}$ and $v=(r'+1)k_{n-2}-1$, we get the following properties on $E\coloneq\bigsqcup_{r\leq m\leq r'}{J_{n,i,m}}$:
        \begin{itemize}
            \item $E$ admits at most $2(\zeta(k'_{n-1})+1)+1$ connected components;
            \item $\max{E}-\min{E}\leq (r'-r+1)k_{n-2}k_{n-1}+(\zeta(k'_{n-1})+2)k_{n-1}$.
        \end{itemize}
    \end{enumerate}

\subsection{Estimates for the cocycles}

\begin{notation} Given a map $\psi\colon I\to\Z$, where $I$ is an interval of $\Z$, we write $\lambda_{\psi}(x)\coloneq\psi(x+1)-\psi(x)$ for every $x\in I^-\coloneqq I\setminus\{\max{I}\}$. The map $\lambda_{\psi}$ is called the \textit{cocycle} of $\psi$. \end{notation}

Recall that $\omega\colon\R_+\to\R_+$ and $\zeta\colon\R_+\to\R_+$ are maps related as in Lemma~\ref{lem:sublinear}

\begin{lemma}\label{lem:CocyclePsi} One can skip steps in the cutting-and-stacking construction of the rank-one systems in such a way that
\begin{itemize}
    \item the assumptions at the beginning of Section~\ref{sec:overview} are still valid;
    \item there exists a summable sequence $(\Gamma_n)$ such that for every $n\geq 3$,
    \begin{equation}\label{eq:InequalityGamma}
        \frac{1}{k_{n+1}}\sum_{i\in [d_{n+1}]}{\sum_{x\in I_{n+1,i}^-}{\omega(|\lambda_{\psi_{n,i}}(x)|)}}\leq \Gamma_n+\frac{1}{k_n}\sum_{\beta\in [d_n]}{\sum_{x\in [k_{n-1}k_{n-2}-1]}{\omega(|\lambda_{\psi^{-1}_{n-1,\beta}}(x)|)}}.
    \end{equation}
\end{itemize}
\end{lemma}

\begin{proof}
Under the assumptions at the beginning of Section~\ref{sec:overview}, we show that \eqref{eq:InequalityGamma} holds with
\begin{align*}
        \Gamma_n\coloneq &\ \frac{k_{n-1}k_{n-2}}{k_n}\ \omega(1)+2\frac{\omega(k_{n-1}k_n)}{k_n} +3\frac{\zeta(k'_{n-1})\ \omega(k_{n-1}k_{n-2})}{k_{n-1}}\\
        &+5\frac{\zeta(k'_{n-1})\ \omega(2k_{n-1}k_{n-2})}{k_{n-1}k_{n-2}}+5\frac{\zeta(k'_{n-1})\ \omega\left(6\zeta(k'_{n-1})k_{n-1}\right )}{k_{n-1}k_{n-2}}\\
        &+5\frac{\zeta(k'_{n-1})\omega(k_n k_{n-1})}{k_n}\\
        &+\frac{\sigma_{n-1}}{k_{n+1}}\omega(1)+5\frac{\sigma_{n-1}}{k_{n+1}}\cdot\frac{\zeta(k'_{n-2})\ \omega(k_{n-1}k_{n-2})}{k_{n-1}}.
    \end{align*}
Assuming it is true, it is straightforward that we can replace $(k_n)$ by a subsequence for $(\Gamma_n)$ to be summable, without affecting any of our standing assumptions. Recall that $(\sigma_{n-1}/k_{n+1})$ is always summable by the definition of a p.m.p. rank-one system (we remind the reader that parameters of the rank-one systems $T^e$ and $T^o$ are supported on even or odd integers, contrary to the convention in Definition~\ref{defr1}).

We now prove \eqref{eq:InequalityGamma}. Let $n\geq 3$, $i\in [d_{n+1}]$. We consider a partition of $I_{n+1,i}^-=\{ik_nk_{n-1},\ldots ,(i+1)k_nk_{n-1}-2\}$ in four sets $\mathcal{J}$, $\partial \mathcal{J}$, $\mathcal{R}$, and $\partial \mathcal{R}$, where
    \begin{itemize}
        \item $\mathcal{J}$ is the set of points $x\in I_{n+1,i}^-$ such that $x$ and $x+1$ lie in the same set $J_{n,i,r}$, for some $r\in [k_{n-1}d_n]$;
        \item $\partial \mathcal{J}$ is the set of points $x\in I_{n+1,i}^-$ such that $x\in J_{n,i,r}$ and $x+1\in J_{n,i,r'}$, with $r\not=r'$;
        \item $\mathcal{R}$ is the set of points $x\in I_{n+1,i}^-$ such that $x$ and $x+1$ lie in the same $R_{n+1,i,\alpha}$, for some $\alpha\in [k_{n-1}]$;
        \item $\partial \mathcal{R}$ is the set of points $x\in I_{n+1,i}^-$ such that $x=\max{R_{n+1,i,\alpha}}$ or $x+1=\min{R_{n+1,i,\alpha}}$, for some $\alpha\in [k_{n-1}]$.
    \end{itemize}
    We also partition $\mathcal{J}=\mathcal{J}^+\sqcup \mathcal{J}^-$ and $\partial \mathcal{J}=\partial \mathcal{J}^+\sqcup \partial \mathcal{J}^-$ as follows.
    \begin{itemize}
        \item Let $x\in \mathcal{J}$, and $r\in [k_{n-1}d_n]$ such that $x$ and $x+1$ lie in the same set $J_{n,i,r}$. We know that $x\in B_{n-1,i,rk_{n-2}+\ell}$ and $x+1\in B_{n-1,i,rk_{n-2}+\ell'}$ for some integers $\ell,\ell'\in [k_{n-2}]$. Then $x\in \mathcal{J}^+$ if and only if $\ell=\ell'$.
        \item Let $x\in\partial \mathcal{J}$, and $r,r'\in [k_{n-1}d_n]$ such that $x\in J_{n,i,r}$ and $x+1\in J_{n,i,r'}$, with $r\neq r'$. We write $r=\alpha d_n+\beta$ and $r=\alpha' d_n+\beta'$ with $\alpha,\alpha'\in [k_{n-1}]$ and $\beta,\beta'\in [d_n]$. Then $x\in\partial \mathcal{J}^+$ if and only if $\alpha=\alpha'$ (in this case $\beta\neq\beta'$), which exactly means that $x$ and $x+1$ lie in the same set $I_{n+1,i,\alpha}$.
    \end{itemize}
    By construction, $|\lambda_{\psi_{n,i}}(x)|=1$ for every $x\in \mathcal{R}$. When $x\in\partial \mathcal{R}$, we consider the coarse bound $\lvert\lambda_{\psi_{n,i}}(x)\rvert\leq k_{n-1}k_n$. Furthermore, we know that $|\mathcal{R}|\leq k_{n-1}r_n\leq k_{n-1}^2k_{n-2}$ and $|\partial \mathcal{R}|\leq 2k_{n-1}$.
    
    If $x\in \mathcal{J}^{-}$, then $x$ and $x+1$ lie in the same set $J_{n,i,\alpha d_n +\beta}$, so $\psi_{n,i}(x), \psi_{n,i}(x+1)\in \alpha k_n +I_{n,\beta}$ and therefore $|\lambda_{\psi_{n,i}}(x)|\leq k_{n-1}k_{n-2}$. By item~\ref{item:3} in Section~\ref{sec:PropBlock}, for a given $r\in [k_{n-1}d_n]$ and a given $\ell\in [k_{n-2}]$, there are at most $\zeta(k'_{n-1}) +2$ points $x$ of $B_{n-1,i,rk_{n-2}+\ell}$ such that $x+1\not\in B_{n-1,i,rk_{n-2}+\ell}$, so 
    \[|\mathcal{J}^-|\leq k_{n-2}k_{n-1}d_n(\zeta(k'_{n-1}) +2)\leq 3\zeta (k'_{n-1})k_n.\]
    
    If $x\in \partial \mathcal{J}^+$, then $x\in J_{n,i,\alpha d_n+\beta}$ and $x+1\in J_{n,i,\alpha d_n+\beta'}$. Denoting $r=\alpha d_n+\beta$ and $r'=\alpha d_n+\beta'$, item~\ref{item:5} in Section~\ref{sec:PropBlock} implies
    \[|\beta'-\beta|\leq 1+\frac{\zeta(k'_{n-1}) +2}{k_{n-2}}\leq 1+3\frac{\zeta(k'_{n-1})}{k_{n-2}}.\]
    Since $\psi_{n,i}(x)\in \alpha k_n + I_{n,\beta}$ and $\psi_{n,i}(x+1)\in \alpha k_n + I_{n,\beta'}$, we get
    \[\lambda_{\psi_{n,i}}(x)\leq |\beta-\beta'|k_{n-1}k_{n-2}\leq k_{n-1}k_{n-2} + 3 \zeta(k'_{n-1})k_{n-1}.\]
    For a given $\alpha\in [k_{n-1}]$ and a given $\beta\in [d_n]$, item~\ref{item:7} in Section~\ref{sec:PropBlock} implies that there are at most $2( \zeta(k'_{n-1})+1)+1\leq 5\zeta(k'_{n-1})$ points $x$ of $J_{n,i,\alpha d_n+\beta}$ such that $x+1\not\in J_{n,i,\alpha d_n+\beta}$, so 
    \[|\partial \mathcal{J}^+|\leq 5\zeta(k'_{n-1})\frac{k_n}{k_{n-2}}.\]
    
    If $x\in\partial \mathcal{J}^-$, then $x\in I_{n+1,i,\alpha}\setminus R_{n+1,i,\alpha}$ and $x+1\in I_{n+1,i,\alpha'} \setminus R_{n+1,i,\alpha'}$, with $\alpha\neq\alpha'$. Item~\ref{item:7} in Section~\ref{sec:PropBlock} implies that there are at most $5\zeta(k'_{n-1})$ points $x$ in $\partial \mathcal{J}^-\cap (I_{n+1,i,\alpha}\setminus R_{n+1,i,\alpha})$. We thus get
    \[
        |\partial \mathcal{J}^-|\leq5\zeta(k'_{n-1})k_{n-1}
    \]
    We consider the coarse bound $|\lambda_{\psi_{n,i}}|\leq k_{n}k_{n-1}$ on $\partial \mathcal{J}^-$.
    
    We finally study the set $\mathcal{J}^+$. We fix $r\in [k_{n-1}d_n]$, written as $\alpha d_n+\beta$ with $\alpha\in [k_{n-1}]$ and $\beta\in [d_n]$. The map $\psi_{n,i}$ restricts to a bijection $J_{n,i,r}=\bigsqcup_{0\leq\ell\leq k_{n-2}-1}{B_{n-1,i,rk_{n-2}+\ell}}\to\alpha k_n +I_{n,\beta}$. Given $\ell\in [k_{n-2}]$ such that $rk_{n-2}+\ell\not\in 2_{n-1,i}$, $\psi_{n,i}=\alpha k_n+\psi^{-1}_{n-1,\beta}\circ \xi_{n,i,r}$ on $B_{n-1,i,rk_{n-2}+\ell}$, where the bijection $\xi_{n,i,r}\colon J_{n,i,r}\to [k_{n-1}k_{n-2}]$ satisfies $\xi_{n,i,r}(x+1)=\xi_{n,i,r}(x)+1$ for every $x\in \mathcal{J}^+\cap J_{n,i,r}$. This implies that for every $\ell\in [k_{n-2}]$ such that $rk_{n-2}+\ell\not\in 2_{n-1,i}$ and for every $x\in \mathcal{J}^+\cap B_{n-1,i,rk_{n-2}+\ell}$, $\lambda_{\psi_{n,i}}(x)=\lambda_{\psi^{-1}_{n-1,\beta}}(\xi_{n,i,r}(x))$.
    
    It remains to treat the case $x\in \mathcal{J}^+\cap B_{n-1,i,rk_{n-2}+\ell}$ for $\ell\in [k_{n-2}]$ such that $rk_{n-2}+\ell\in 2_{n-1,i}$. By construction, $\psi_{n,i}$ restricts to an increasing bijection
    \[
        \bigsqcup_{\substack{\ell\in [k_{n-2}]\\rk_{n-2}+\ell\in 2_{n-1,i}}}B_{n-1,i,r k_{n-2}+\ell}\longrightarrow\bigsqcup_{\substack{\ell\in [k_{n-2}]\\rk_{n-2}+\ell\in 2_{n-1,i}}}(\alpha k_n+I_{n,\beta,\ell})
    \]
    (see Remark~\ref{rem:ToSumUp}). Therefore, $\lambda_{\psi_{n,i}}(x)=1$ for every $x\in \bigsqcup_{\ell\in [k_{n-2}],\ rk_{n-2}+\ell\in 2_{n-1,i}}{B_{n-1,i,r k_{n-2}+\ell}}$ (at most $|(rk_{n-2}+[k_{n-2}])\cap 2_{n-1,i}|\cdot k_{n-1}$ elements), except when $\psi_{n,i}(x)$ and $\psi_{n,i}(x+1)$ do not lie in the same connected component of the target set (at most $|(rk_{n-2}+[k_{n-2}])\cap 2_{n-1,i}|\cdot (2(\zeta(k'_{n-2})+1)+1)$ elements by item~\ref{item:7} in Section~\ref{sec:PropBlock}), in which case $\lambda_{\psi_{n,i}}(x)\leq\operatorname{diam}(\alpha k_n+I_{n,\beta})= k_{n-1}k_{n-2}$.

    Writing $r=\alpha d_n+\beta$, we thus get

    \begin{align*}
        \sum_{x\in \mathcal{J}^+\cap J_{n,i,r}}{\omega(|\lambda_{\psi_{n,i}}(x)|)}
       \leq &\ |(rk_{n-2}+[k_{n-2}])\cap 2_{n-1,i}|\cdot k_{n-1}\ \omega(1)\\
       &+|(rk_{n-2}+[k_{n-2}])\cap 2_{n-1,i}|\cdot 5\zeta(k'_{n-2})\ \omega(k_{n-1}k_{n-2})\\
       &+\sum_{x\in [k_{n-1}k_{n-2}-1]}{\lambda_{\psi^{-1}_{n-1,\beta}}(x)},
    \end{align*}
    so
    \begin{align*}
        \sum_{x\in \mathcal{J}^+}{\omega(|\lambda_{\psi_{n,i}}(x)|)}
        = &\sum_{\alpha\in [k_{n-1}],\ \beta\in [d_n]}{\sum_{x\in \mathcal{J}^+\cap J_{n,i,\alpha d_n+\beta}}{\omega(|\lambda_{\psi_{n,i}}(x)|)}}\\
        \leq &\ |[d_nk_{n-1}k_{n-2}]\cap 2_{n-1,i}|\cdot k_{n-1}\ \omega (1)\\
        &+5\ |[d_nk_{n-1}k_{n-2}]\cap 2_{n-1,i}|\cdot\zeta(k'_{n-2})\ \omega(k_{n-1}k_{n-2})\\
        &+k_{n-1}\sum_{\beta\in [d_n]}{\sum_{x\in [k_{n-1}k_{n-2}-1]}{\omega(|\lambda_{\psi^{-1}_{n-1,\beta}}(x)|)}}.
    \end{align*}
    Gathering all these estimates, we have
    \begin{align*}
        \sum_{x\in I_{n+1,i}^-}{\omega(|\lambda_{\psi_{n,i}}(x)|)}
        \leq &\ k_{n-1}^2k_{n-2}\ \omega(1)\\
        &+2k_{n-1}\ \omega(k_{n-1}k_n)\\
        &+3\zeta (k'_{n-1})k_n\ \omega(k_{n-1}k_{n-2})\\
        &+5\frac{\zeta(k'_{n-1})k_n}{k_{n-2}}\ \omega(k_{n-1}k_{n-2}+3\zeta(k'_{n-1}) k_{n-1})\\
        &+5\zeta(k'_{n-1})k_{n-1}\ \omega(k_n k_{n-1})\\
        &+|[d_nk_{n-1}k_{n-2}]\cap 2_{n-1,i}|\cdot k_{n-1}\ \omega(1)\\
        &+5|[d_nk_{n-1}k_{n-2}]\cap 2_{n-1,i}|\cdot\zeta(k'_{n-2})\ \omega(k_{n-1}k_{n-2})\\
        &+k_{n-1}\sum_{\beta\in [d_n]}{\sum_{x\in [k_{n-1}k_{n-2}-1]}{\omega(|\lambda_{\psi^{-1}_{n-1,\beta}}(x)|)}}.
    \end{align*}
    
    We now sum over $i\in [d_{n+1}]$. Note that $\sum_i{|[d_nk_{n-1}k_{n-2}]\cap 2_{n-1,i}|}$ counts the number of blocks of second type in $F_{n+1}$. Since these blocks are each composed of $k_{n-1}$ spacers, we get $\sum_i{|[d_nk_{n-1}k_{n-2}]\cap 2_{n-1,i}|\ k_{n-1}}\leq \sigma_{n-1}$. This implies
    \begin{align*}
        \frac{1}{k_{n+1}}\sum_{i\in [d_{n+1}]}{\sum_{x\in I_{n+1,i}^-}{\omega(|\lambda_{\psi_{n,i}}(x)|)}}\leq &\ \frac{k_{n-1}k_{n-2}}{k_n}\ \omega(1)+2\frac{\omega(k_{n-1}k_n)}{k_n} +3\frac{\zeta(k'_{n-1})\ \omega(k_{n-1}k_{n-2})}{k_{n-1}}\\
        &+5\frac{\zeta(k'_{n-1})\ \omega(k_{n-1}k_{n-2}+3\zeta(k'_{n-1}) k_{n-1})}{k_{n-1}k_{n-2}}\\
        &+5\frac{\zeta(k'_{n-1})\ \omega(k_n k_{n-1})}{k_n}\\
        &+\frac{\sigma_{n-1}}{k_{n+1}}\omega(1)+5\frac{\sigma_{n-1}}{k_{n+1}}\cdot\frac{\zeta(k'_{n-2})\ \omega(k_{n-1}k_{n-2})}{k_{n-1}}\\
        &+\frac{1}{k_n}\sum_{\beta\in [d_n]}{\sum_{x\in [k_{n-1}k_{n-2}-1]}{\omega(|\lambda_{\psi^{-1}_{n-1,\beta}}(x)|)}}.
    \end{align*}
    Since $\omega\colon\R_+\to\R_+$ is increasing, $\omega(a+b)\leq \omega(2a)+\omega(2b)$ holds for all positive real numbers $a,b$. We thus get
    \[\omega(k_{n-1}k_{n-2}+3\zeta(k'_{n-1}) k_{n-1})\leq \omega(2k_{n-1}k_{n-2})+\omega\left(6\zeta(k'_{n-1})k_{n-1}\right),\]
    which yields the desired bound.
\end{proof}

We now consider the maps $\varphi_n\colon F_{n+1}\to F_n$ defined by
\[
    \forall x\in F_{n+1},\ \varphi_n(x)=\begin{cases}
                    \psi_{n,i}(x)\bmod k_n&\text{ if }x\in I_{n+1,i}\text{ for some }i\in[d_{n+1}]\\
                    x\bmod k_n&\text{ if }x\in I_{n+1,\ast}
                \end{cases}
\]
for $n\geq1$.

\begin{lemma}\label{lem:CocyclePhi}
    For every $n\geq 3$, we have
    \begin{align*}
        \frac{1}{k_{n+1}}\sum_{x\in [k_{n+1}-1]}{\omega(|\lambda_{\varphi_{n}}(x)|)}\leq &\ \frac{1}{k_nk_{n-1}}\omega(k_n)+\frac{k_nk_{n-1}}{k_{n+1}}\omega(k_n)\\
        &+\Gamma_n+\frac{1}{k_n}\sum_{\beta\in [d_n]}{\sum_{x\in [k_{n-1}k_{n-2}-1]}{\omega(|\lambda_{\psi^{-1}_{n-1,\beta}}(x)|)}}.
    \end{align*}
\end{lemma}

Note that the sequence $\left (\frac{1}{k_nk_{n-1}}\omega(k_n)\right)_{n\geq 1}$ is bounded above since $\omega$ is sublinear.

\begin{proof}
    Given $i\in [d_{n+1}]$, we deduce from the definition of $\varphi_n$ that $\lambda_{\varphi_n}(x)=\lambda_{\psi_{n,i}}(x)$ for all points $x\in I_{n+1,i}^-$ with $\psi_{n,i}(x+1)$ and $\psi_{n,i}(x)$ lying in the same interval $\alpha k_n+[k_n]$ for some $\alpha\in [k_{n-1}]$. For the points $x$ which do not satisfy this property, we already used the crude bound $\lambda_{\psi_{n,i}}(x)\leq k_{n-1}k_n$ in the last proof, a bound which still works for $\lambda_{\varphi_{n}}(x)$. We can thus use the estimate of Lemma~\ref{lem:CocyclePsi} and add the contributions of $\lambda_{\varphi_{n}}(x)$ for all $x= (i+1)k_nk_{n-1}-1$, $i\in [d_{n+1}-1]$ and for $x\in I_{n+1,\ast}$. In both cases, we simply use the coarse bound $\lambda_{\varphi_{n}}(x)\leq k_n$. Gathering all these estimates, we get the desired bound.
\end{proof}

\begin{lemma}\label{lem:CocyclePsiInverse}
    One can skip steps in the cutting-and-stacking construction of the rank-one systems so that
    \begin{itemize}
        \item the assumptions at the beginning of Section~\ref{sec:overview} are still valid;
        \item the sequence $(\Gamma_n)$ in Lemma~\ref{lem:CocyclePsi} is summable;
        \item there exists a summable sequence $(\Delta_n)$ such that for every $n\geq 3$,
    \begin{equation}\label{eq:InequalityDelta}
        \frac{1}{k_{n+1}}\sum_{i\in [d_{n+1}]}{\sum_{x\in [k_n k_{n-1}-1]}{\omega(|\lambda_{\psi^{-1}_{n,i}}(x)|)}}
        \leq \Delta_n+\ \frac{1}{k_{n}}\sum_{\beta\in [d_n]}{\sum_{x\in I_{n,\beta}^-}{\omega(|\lambda_{\psi_{n-1,\beta}}(x)|)}}.
    \end{equation}
    \end{itemize}
\end{lemma}

\begin{proof}
    Under the assumptions at the beginning of Section~\ref{sec:overview}, we prove that \eqref{eq:InequalityDelta} holds with
    \begin{align*}
        \Delta_n\coloneq&\ \frac{k_{n-1}k_{n-2}}{k_n}\ \omega(1)+2\frac{\omega(k_{n-1}k_n)}{k_n}\\
        &+\ \frac{\sigma_{n-1}}{k_{n+1}}\ \omega(1)+\frac{\omega(2k_{n-1})}{k_{n-1}}\\
        &+\ 5\frac{\zeta(k'_{n-2})\ \omega(4k_{n-2}k_{n-1})}{k_{n-1}}+5\frac{\zeta(k'_{n-2})\ \omega(6\zeta(k'_{n-1})k_{n-1})}{k_{n-1}}.
    \end{align*}
    Once again, replacing $(k_n)$ by a subsequence for $(\Delta_n)$ to be summable does not interfere with our standing assumptions, nor with summability of $(\Gamma_n)$.
    
    To prove \eqref{eq:InequalityDelta} we proceed as in the proof of Lemma~\ref{lem:CocyclePsi}. Let $n\geq 3$ and $i\in [d_{n+1}]$. We partition $[k_nk_{n-1}-1]$ in five sets  $\mathcal{J}_1$, $\mathcal{J}_2$, $\partial \mathcal{J}$, $\mathcal{R}$, and $\partial \mathcal{R}$, where
    \begin{itemize}
        \item $\mathcal{J}_1$ is the set of points $x\in [k_nk_{n-1}-1]$ such that $x$ and $x+1$ lie in the same set $\alpha k_n +I_{n,\beta,\ell}$, for some $\alpha\in [k_{n-1}]$, $\beta\in [d_n]$, $\ell\in [k_{n-2}]$ such that $(\alpha d_n+\beta)k_{n-2}+\ell\not\in 2_{n-1,i}$;
        \item $\mathcal{J}_2$ is the set of points $x\in [k_nk_{n-1}-1]$ such that $x$ and $x+1$ lie in 
        \[\bigsqcup_{\substack{\alpha\in [k_{n-1}]\\ \beta\in [d_n]}}{\bigsqcup_{\substack{l\in [k_{n-2}]\\ (\alpha d_n +\beta)k_{n-2}+\ell\in 2_{n-1,i}}}{(\alpha k_n+I_{n,\beta,\ell})}};\]
        \item $\partial \mathcal{J}$ is the set of points $x\notin\mathcal{J}_1\sqcup\mathcal{J}_2$ such that $x$ and $x+1$ lie in
        \[
            \bigsqcup_{\alpha\in[k_{n-1}]}(\alpha k_n+([k_n]\setminus I_{n,*}));
        \]
        \item $\mathcal{R}$ is the set of points $x$ such that $x$ and $x+1$ lie in the same set $\alpha k_n+I_{n,\ast}$, for some $\alpha\in [k_{n-1}]$;
        \item $\partial \mathcal{R}$ is the set of points $x$ such that $x=\max{(\alpha k_n+I_{n,\ast})}$ or $x+1=\min{(\alpha k_n+I_{n,\ast})}$, for some $\alpha\in [k_{n-1}]$.
    \end{itemize}
    
    By definition, $\lambda_{\psi_{n,i}^{-1}}(x)=1$ for every $x\in \mathcal{R}$, and $|\mathcal{R}|\leq k_{n-1}r_n\leq k_{n-1}^2k_{n-2}$. For $x\in\partial \mathcal{R}$, we use the coarse bound $\lambda_{\psi_{n,i}^{-1}}(x)\leq k_{n-1}k_n$. Note that $|\partial \mathcal{R}|\leq 2k_{n-1}$.
    
    For $\mathcal{J}_2$, note that $\psi_{n,i}$ restricts to an increasing bijection from the union of blocks of the second type in $I_{n+1,i}$ to the displayed set in the definition of $\mathcal{J}_2$. Recall first that blocks of the second type are only composed of spacers, and second that $\sigma_{n,j}>0$ for every $j\in\{1,\ldots,q_{n-1}-1\}$. Therefore, for $x\in \mathcal{J}_2$ (at most $|[d_n k_{n-1}k_{n-2}]\cap 2_{n-1,i}|\cdot k_{n-1}$ points), $\lambda_{\psi_{n,i}^{-1}}(x)=1$, except when there is a jump over a block of the first type in passing from $\psi_{n,i}^{-1}(x)$ to $\psi_{n,i}^{-1}(x+1)$. The latter occurs for at most $|C_{n-1,i}|$ points, in which case we use the bound $\lambda_{\psi_{n,i}^{-1}}(x)\leq k_{n-1}+1$.

    Let $x\in \mathcal{J}_1$. There exist $\alpha\in [k_{n-1}]$, $\beta\in [d_n]$, and $\ell\in [k_{n-2}]$ such that $rk_{n-2}+\ell\not\in 2_{n-1,i}$ and $x,x+1\in\alpha k_n +I_{n,\beta,\ell}$, where $r\coloneq \alpha d_n+\beta$. By definition of the (restricted) bijection $\psi_{n,i}\colon B_{n-1,i,rk_{n-2}+\ell}\to\alpha k_n+I_{n,\beta,\ell}$, and since $\xi_{n,i,r}$ is a translation on $B_{n-1,i,rk_{n-2}+\ell}$, we have $\lambda_{\psi_{n,i}^{-1}}(x)=\lambda_{\psi_{n-1,\beta}}(x-\alpha k_n)$. This implies that
    \[\sum_{x\in \mathcal{J}_1}{\omega(|\lambda_{\psi_{n,i}^{-1}}(x)|)}\leq k_{n-1}\sum_{\beta\in [d_n]}{\sum_{x\in I_{n,\beta}^-}{\omega(|\lambda_{\psi_{n-1,\beta}}(x)|)}}.\]
    
    We finally treat the case of $x\in\partial \mathcal{J}$. There exist $\alpha,\alpha'\in [k_{n-1}]$, $\beta,\beta'\in [d_n]$, and $\ell,\ell'\in [k_{n-2}]$ such that $x\in \alpha k_n+I_{n,\beta,\ell}$ and $x+1\in \alpha' k_n+I_{n,\beta',\ell'}$. The sets $\alpha k_n+I_{n,\beta}$ and $\alpha' k_n+I_{n,\beta'}$ are necessarily consecutive, so $\alpha=\alpha'$, and denoting $r=\alpha d_n+\beta$ and $r'=\alpha d_n+\beta'$, we have $|r-r'|\leq 1$. By definition, $\psi_{n,i}^{-1}(x)\in J_{n,i,\alpha d_n +\beta}$ and $\psi_{n,i}^{-1}(x+1)\in J_{n,i,\alpha' d_n +\beta'}$, so item~\ref{item:7} provides the bound $\lambda_{\psi_{n,i}^{-1}}(x)\leq 2k_{n-2}k_{n-1}+3\zeta(k'_{n-1})k_{n-1}$. Finally, an element of $\partial \mathcal{J}$ is completely determined by the values of $\alpha,\beta,\ell$ and by the connected component of $\alpha k_n +I_{n,\beta,\ell}$ containing it. The definition of $I_{n,\beta,\ell}$ and item~\ref{item:7} provide the number of connected components, so
    \[|\partial \mathcal{J}|\leq k_{n-1} d_n k_{n-2}\ 5\zeta(k'_{n-2})\leq 5k_n\zeta(k'_{n-2}).\]

    We gather all the estimates
    \begin{align*}
        \sum_{x\in [k_n k_{n-1}-1]}{\omega(|\lambda_{\psi^{-1}_{n,i}}(x)|)}\leq &\ k_{n-1}^2k_{n-2}\ \omega(1) + 2k_{n-1}\ \omega(k_{n-1}k_n)\\
        &+\ |[d_nk_{n-1}k_{n-2}]\cap 2_{n-1,i}|\cdot k_{n-1}\ \omega(1)+\ |C_{n-1,i}|\ \omega(2k_{n-1})\\
        &+\ 5k_n\zeta(k'_{n-2})\ \omega(2k_{n-2}k_{n-1}+3\zeta(k'_{n-1})k_{n-1})\\
        &+\ k_{n-1}\sum_{\beta\in [d_n]}{\sum_{x\in I_{n,\beta}^-}{\omega(|\lambda_{\psi_{n-1,\beta}}(x)|)}},
    \end{align*}
    and sum over $i\in [d_{n+1}]$
    \begin{align*}
        \frac{1}{k_{n+1}}\sum_{i\in [d_{n+1}]}{\sum_{x\in [k_n k_{n-1}-1]}{\omega(|\lambda_{\psi^{-1}_{n,i}}(x)|)}}
        \leq &\ \frac{k_{n-1}k_{n-2}}{k_n}\ \omega(1)+\frac{2}{k_n}\ \omega(k_{n-1}k_n)\\
        &+\ \frac{\sigma_{n-1}}{k_{n+1}}\ \omega(1)+\frac{q_{n-1}}{k_{n+1}}\ \omega(2k_{n-1})\\
        &+\ \frac{5\zeta(k'_{n-2})}{k_{n-1}}\ \omega(2k_{n-2}k_{n-1}+3\zeta(k'_{n-1})k_{n-1})\\
        &+\ \frac{1}{k_{n}}\sum_{\beta\in [d_n]}{\sum_{x\in I_{n,\beta}^-}{\omega(|\lambda_{\psi_{n-1,\beta}}(x)|)}}.
    \end{align*}
    We finally use the inequalities $\omega(2k_{n-2}k_{n-1}+3\zeta(k'_{n-1})k_{n-1})\leq \omega(4k_{n-2}k_{n-1})+\omega(6\zeta(k'_{n-1})k_{n-1})$ and $q_{n-1}/k_{n+1}\leq 1/k_{n-1}$ to get the desired bound.
\end{proof}

From Lemmas~\ref{lem:CocyclePsi},~\ref{lem:CocyclePhi},~\ref{lem:CocyclePsiInverse}, we deduce the following.

\begin{corollary}\label{cor:CocyclePhiSeries}
    If the assumptions of Lemma~\ref{lem:CocyclePsiInverse} are satisfied, and if the sequence
    \[\left (\frac{k_nk_{n-1}}{k_{n+1}}\omega(k_n)\right)_{n\geq 1}\]
    is bounded above, then
    \[\sup_{n\geq 1}{\frac{1}{k_{n+1}}\sum_{x\in [k_{n+1}-1]}{\omega(|\lambda_{\varphi_{n}}(x)|)}}< +\infty.\]
\end{corollary}

\subsection{Proof of Theorem~\ref{TheoremA}}

For every $n\geq 0$, we denote by
\[p_n\colon F_n+C_n\to F_n\]
the canonical projection defined by $p_n(f_n+c_n)=f_n$ for every $f_n\in F_n$ and $c_n\in C_n$. The map is well defined since the sets $F_{n-1}+c_n$, for $c_n\in C_n$, are pairwise disjoint.

\begin{lemma}\label{lem:projection}
    For every $n\geq 1$, the map $\varphi_n\circ\varphi_{n+1}\colon F_{n+2}\to F_n$ restricts to $p_n$ on $F_n+C_n$.
\end{lemma}

\begin{proof}
    Let $c\in C_n$. The set $F_n+c$ lies in $I_{n+2,i}$ for some $i\in [d_{n+2}]$ (in other words, $c\in C_{n,i}$). Specifically, it is included in $J_{n+1,i,r}$ for some $r\in [k_n d_{n+1}]$. We write $r=\alpha d_{n+1}+\beta$ with $\alpha\in [k_n]$ and $\beta\in [d_{n+1}]$. Finally, $F_n+c$ is equal to the block $B_{n,i,r k_{n-1}+\ell}$ for some $\ell\in [k_{n-1}]$. By construction, $\psi_{n+1,i}=\alpha k_{n+1}+\psi^{-1}_{n,\beta}(\cdot - c +\ell k_n)$ on $F_n+c$. By definition of $\varphi_{n+1}$, this implies that the map $\varphi_{n+1}$ coincides with $\psi^{-1}_{n,\beta}(\cdot - c +\ell k_n)$ on $F_n+c$, so it takes values in $I_{n+1,\beta}$. Again by definition, we get
    \[\varphi_n\circ\varphi_{n+1}=\left (\psi_{n,\beta}\circ\psi^{-1}_{n,\beta}(\cdot - c +\ell k_n)\right )\bmod k_n=\id_{F_n+c} -c=p_n.\]
    on $F_n+c$, concluding the proof.
\end{proof}

\begin{lemma}\label{lem:projection2}
        If the series $\sum{(k_nk_{n+1}/k_{n+2})}$ converges, then for every $n\geq 1$ there exists $A_{n+2}\subset F_{n}+C_{n}$ such that
        \begin{align*}
            \left (\frac{|F_{n+2}\setminus A_{n+2}|}{|F_{n+2}|}\right)_{n\geq 0}
        \end{align*}
        is summable and $p_{n-1}\circ\varphi_{n+1}=\varphi_{n-1}\circ p_n$ on $A_{n+2}$.
\end{lemma}

\begin{proof}
    From Lemma~\ref{lem:projection}, we deduce that for every $x\in A_{n+2}\coloneq (F_n+C_n)\cap\varphi_{n+1}^{-1}(F_{n-1}+C_{n-1})$, we have $p_{n-1}\circ\varphi_{n+1}=\varphi_{n-1}\circ\varphi_n\circ\varphi_{n+1}=\varphi_{n-1}\circ p_n$. To prove that $(|F_{n+2}\setminus A_{n+2}|/|F_{n+2}|)$ is summable, we only have to prove the summability of $(|F_{n+2}\setminus \varphi_{n+1}^{-1}(F_{n-1}+C_{n-1})|/|F_{n+2}|)$ since the series $\sum{(\sigma_{n}/k_{n+2})}$ converges. It is straightforward that $|\varphi_{n+1}^{-1}(A)|\leq k_n d_{n+2}|A|+r_{n+2}$ for every subset $A\subset F_{n+2}$. This implies that
    \begin{align*}
        \frac{|F_{n+2}\setminus \varphi_{n+1}^{-1}(F_{n-1}+C_{n-1})|}{|F_{n+2}|}&\leq\frac{|\varphi_{n+1}^{-1}(F_{n+1}\setminus (F_{n-1}+C_{n-1}))|}{|F_{n+2}|}\\
        &\leq \frac{k_n d_{n+2}|F_{n+1}\setminus (F_{n-1}+C_{n-1})|+r_{n+2}}{k_{n+2}}\\
        &\leq \frac{k_n \frac{k_{n+2}}{k_{n+1}k_n}|F_{n+1}\setminus (F_{n-1}+C_{n-1})|+k_{n+1}k_n}{k_{n+2}}\\
        &\leq \frac{|F_{n+1}\setminus (F_{n-1}+C_{n-1})|}{|F_{n+1}|}+\frac{k_{n+1}k_n}{k_{n+2}}.
    \end{align*}
    The upper bound is summable, so we are done.
\end{proof}

\begin{lemma}\label{lem:ThisIsAnOE}
    One can skip steps in the cutting-and-stacking construction of the rank-one systems in such a way that
    \begin{itemize}
        \item the assumptions at the beginning of Section~\ref{sec:overview} are still valid;
        \item the sequence $(\Gamma_n)$ in Lemma~\ref{lem:CocyclePsi} is summable;
        \item the sequence $(\Delta_n)$ in Lemma~\ref{lem:CocyclePsiInverse} is summable;
        \item for every $n\geq 3$, there exists $A'_{n+1}\subset F_{n-1}+C_{n-1}$ such that the sequence
    \begin{align*}
        \left (\frac{|F_{n+1}\setminus A'_{n+1}|}{|F_{n+1}|}\right)_{n\geq 3}
    \end{align*}
    is summable, and $\lambda_{\varphi_n}(x)=\lambda_{\varphi_{n-2}}(p_{n-1}(x))$ for every $x\in A'_{n+1}$.
    \end{itemize}
\end{lemma}

\begin{proof}
    Let $c\in C_{n-1}$. There exist $i\in [d_{n+1}]$, $\alpha\in [k_{n-1}]$, $\beta\in [d_n]$, and $\ell\in [k_{n-2}]$ such that $F_{n-1}+c=B_{n-1,i,r k_{n-2}+\ell}$, where $r\coloneq \alpha d_n+\beta$. Recall that $\xi_{n,i,r}$ is the translation by $-c+\ell k_{n-1}$, from $B_{n-1,i,r k_{n-2}+\ell}$ to $\ell k_{n-1}+[k_{n-1}]$. We denote by $A'_{n+1,c}$ the set of points $x$ satisfying the following properties:
    \begin{enumerate}[label=(P\arabic*)]
        \item\label{item:LemmaThisIsAnOE1} $x,x+1\in F_{n-1}+c$;
        \item\label{item:LemmaThisIsAnOE2} there exist $\beta'\in [d_{n-1}]$ and $\ell'\in [k_{n-3}]$ such that $(\ell d_{n-1}+\beta')k_{n-3}+\ell'\not\in 2_{n-2,\beta}$ and $\xi_{n,i,r}(x),\xi_{n,i,r}(x)+1\in \ell k_{n-1}+ I_{n-1,\beta',\ell'}$;
        \item\label{item:LemmaThisIsAnOE3} $\lambda_{\varphi_{n-2}}(p_{n-1}(x))=\lambda_{\psi_{n-2,\beta'}}(p_{n-1}(x))$.
    \end{enumerate}
    Let $x\in A'_{n+1,c}$ and $\beta',\ell'$ as above. Using the study of $\mathcal{J}^+$ in the proof of Lemma~\ref{lem:CocyclePsi}, \ref{item:LemmaThisIsAnOE1} implies that $\lambda_{\psi_{n,i}}(x)=\lambda_{\psi_{n-1,\beta}^{-1}}(\xi_{n,i,r}(x))$. Using the study of $\mathcal{J}_1$ in the proof of Lemma~\ref{lem:CocyclePsiInverse}, \ref{item:LemmaThisIsAnOE2} implies that $\lambda_{\psi_{n-1,\beta}^{-1}}(\xi_{n,i,r}(x))=\lambda_{\psi_{n-2,\beta'}}(\xi_{n,i,r}(x)-\ell k_{n-1})$, where $\xi_{n,i,r}(x)-\ell k_{n-1}=x-c=p_{n-1}(x)$. By~\ref{item:LemmaThisIsAnOE1} and the discussion in the proof of Lemma~\ref{lem:CocyclePhi}, we have $\lambda_{\phi_{n}}(x)=\lambda_{\psi_{n,i}}(x)$, so we get $\lambda_{\phi_{n}}(x)=\lambda_{\phi_{n-2}}(p_{n-1}(x))$ by~\ref{item:LemmaThisIsAnOE3}.
    We have proved that for every $x\in A'_{n+1}\coloneq\bigsqcup_{c\in C_{n-1}}{A'_{n+1,c}}$, we have the desired equality between the cocycles. It remains to find an estimate for $\frac{|F_{n+1}\setminus A'_{n+1}|}{|F_{n+1}|}$. We have
    \begin{align*}
        \frac{|F_{n+1}\setminus A'_{n+1}|}{|F_{n+1}|}&=\frac{\sigma_{n-1}}{k_{n+1}}+\frac{1}{k_{n+1}}\sum_{c\in C_{n-1}}{|(F_{n-1}+c)\setminus A'_{n+1,c}|},
    \end{align*}
    so we need an estimate on each $|(F_{n-1}+c)\setminus A'_{n+1,c}|$. Let $c\in C_{n-1}$, $i$, $\ell$, and $r=\alpha d_n+\beta$ as in the beginning of the proof. We examine the set $F_{n-1}+c$:
    \begin{itemize}
        \item There is exactly one point which does not satisfy~\ref{item:LemmaThisIsAnOE1}.
        \item There are $|(\ell d_{n-1}k_{n-3}+[d_{n-1}k_{n-3}])\cap 2_{n-2,\beta}|$ couples $(\beta',\ell')$ such that $(\ell d_{n-1}+\beta')k_{n-3}+\ell'\in 2_{n-2,\beta}$, and for each of them, at most $k_{n-2}$ points $x$ such that $\xi_{n,i,r}(x)\in\ell k_{n-1}+ I_{n-1,\beta',\ell'}$. Moreover, there are $r_{n-1}$ points $x$ such that $\xi_{n,i,r}(x)\in\ell k_{n-1}+ I_{n-1,\ast}$. Finally, given a couple $(\beta',\ell')$ such that $(\ell d_{n-1}+\beta')k_{n-3}+\ell'\not\in 2_{n-2,\beta}$, there is at most one point which does not satisfy~\ref{item:LemmaThisIsAnOE2} relative to this couple.
        \item By the proof of Lemma~\ref{lem:CocyclePsi} and the discussion in the proof of Lemma~\ref{lem:CocyclePhi}, we know that~\ref{item:LemmaThisIsAnOE3} is satisfied at least for the points $x\in F_{n-1}+c$ such that $p_{n-1}(x)\in F_{n-3}+C_{n-3}$ and moreover $p_{n-1}(x)$ and $p_{n-1}(x)+1$ lie in the same $F_{n-3}+c'$. Therefore, at most $\sigma_{n-3}+q_{n-3}$ points of $F_{n-1}+c$ do not satisfy~\ref{item:LemmaThisIsAnOE3}.
    \end{itemize}

    We thus get the following estimate:
    \begin{align*}
        |(F_n+c)\setminus A'_{n+1,c}|\leq&\  1\\
        &+|(\ell d_{n-1}k_{n-3}+[d_{n-1}k_{n-3}])\cap 2_{n-2,\beta}|\cdot k_{n-2}+k_{n-2}k_{n-3}+d_{n-1}k_{n-3}\\
        &+\sigma_{n-3}+q_{n-3}.
    \end{align*}
    We sum over $c\in C_{n-1}$, which means that we sum over $i\in [d_{n+1}]$, $\alpha\in [k_{n-1}]$, $\beta\in [d_n]$, and $\ell\in [k_{n-2}]$ such that $(\alpha d_n+\beta)k_{n-2}+\ell\not\in 2_{n-1,i}$. We first treat the second summand. The quantity
    \[\sum_{\ell\in [k_{n-2}]}{|(\ell d_{n-1}k_{n-3}+[d_{n-1}k_{n-3}])\cap 2_{n-2,\beta}|}=|[k_{n-2}d_{n-1}k_{n-3}]\cap 2_{n-2,\beta}|\]
    is exactly equal to the number of indices $j\in [k_{n-2}d_{n-1}k_{n-3}]$ such that $B_{n-2,\beta,j}$ is a block of the second type. Multiplying by $k_{n-2}$, namely the cardinality of each $B_{n-2,\beta,j}$, we get the number of spacers lying in a block of the second type in $I_{n-1,\beta}$. This implies that we can bound the sum over $\beta\in [d_n]$ by $\sigma_{n-2}$. We finally have:
    \begin{align*}
        \sum_{c\in C_{n-1}}{|(F_n+c)\setminus A'_{n+1,c}|}\leq & \ |C_{n-1}|\\
        &+d_{n+1}k_{n-1}\cdot \sigma_{n-2}+d_{n+1}k_{n-1}d_n k_{n-2}\cdot (k_{n-2}k_{n-3}+d_{n-1}k_{n-3})\\
        &+d_{n+1}k_{n-1}d_n k_{n-2}\cdot (\sigma_{n-3}+q_{n-3})\\
        \leq &\ \frac{k_{n+1}}{k_{n-1}}\\
        &+\frac{k_{n+1}\sigma_{n-2}}{k_n}+\frac{k_{n+1}k_{n-2}k_{n-3}}{k_{n-1}}+\frac{k_{n+1}}{k_{n-2}}\\
        &+\frac{k_{n+1}\sigma_{n-3}}{k_{n-1}}+\frac{k_{n+1}}{k_{n-3}}.
    \end{align*}
    This enables us to get
    \[\frac{|F_{n+1}\setminus A'_{n+1}|}{|F_{n+1}|}\leq\frac{\sigma_{n-1}}{k_{n+1}}+\frac{\sigma_{n-2}}{k_n}+\frac{\sigma_{n-3}}{k_{n-1}}+\frac{1}{k_{n-1}}+\frac{k_{n-2}k_{n-3}}{k_{n-1}}+\frac{1}{k_{n-2}}+\frac{1}{k_{n-3}},\]
    where we can assume without loss of generality that the upper bound is summable over $n$.
\end{proof}

\begin{proof}[Proof of Theorem~\ref{TheoremA}]
Let $T^e$ and $T^o$ be two rank-one systems, and $\omega\colon\R_+\to\R_+$ be a sublinear map. By~\cite[Lemma~2.12]{carderiBelinskayaTheoremOptimal2023}, we can assume without loss of generality that $\omega$ is increasing.

For each system, we fix a cutting-and-stacking construction and denote the associated parameters by $(k'_{2n})_{n\geq 0}$, $(q'_{2n})_{n\geq 0}$, $(\sigma'_{2n,i})_{n\geq 0}$ and $(k'_{2n+1})_{n\geq 0}$, $(q'_{2n+1})_{n\geq 0}$, $(\sigma'_{2n+1,i})_{n\geq 0}$ respectively. Skipping steps, we assume that $\sum_{n\geq 1}{(\sigma'_n/k'_{n+1})}\leq 1/2$ as in Lemma~\ref{lem:ProperlyChosenCuttingAndStacking1}. We set $k_n\coloneq k'_n-\left\lfloor k'_n/\zeta(k'_n)\right\rfloor$. By our preparatory work, we can assume without loss of generality that:
\begin{enumerate}
    \item\label{item:ProofTheorem1} $k_{n+2}>k_{n+1}>k_n$ and $k_{n+2}>k_{n+1}k_n$;
    \item\label{item:ProofTheorem2} the remaining assumptions of Lemma~\ref{lem:ProperlyChosenCuttingAndStacking1} are satisfied, as well as that of Lemma~\ref{lem:ProperlyChosenCuttingAndStacking2};
    \item\label{item:ProofTheorem3} $(\Gamma_n+\Delta_n)_{n\geq1}$ is summable and $((k_nk_{n-1}/k_{n+1})\cdot\omega(k_n))_{n\geq 1}$ is bounded (see Lemmas~\ref{lem:CocyclePsi} and~\ref{lem:CocyclePsiInverse});
    \item\label{item:ProofTheorem4} subsets $A_n$ and $A'_n$ as in Lemmas~\ref{lem:projection2} and~\ref{lem:ThisIsAnOE} exist.
\end{enumerate}

Assumption~\eqref{item:ProofTheorem1} guarantees that the construction carried out in Section~\ref{sec:Construction} is well-defined. Assumption~\eqref{item:ProofTheorem2} implies that we can complete $(k_{2n})_{n\geq 0}$ and $(k_{2n+1})_{n\geq 0}$ with spacing and cutting parameters $(\sigma_{2n})_{n\geq 0}$, $(q_{2n})_{n\geq 0}$ and $(\sigma_{2n+1})_{n\geq 0}$, $(q_{2n+1})_{n\geq 0}$ respectively, in such a way that they still describe cutting-and-stacking constructions of $T^e$ and $T^o$, and that the conclusions of Lemmas~\ref{lem:ProperlyChosenCuttingAndStacking1} and~\ref{lem:ProperlyChosenCuttingAndStacking2} hold. We needed the latter for the estimates leading to Corollary~\ref{cor:CocyclePhiSeries}.

We now consider $T^e\colon (X^e,\mu^e)\to (X^e,\mu^e)$ and $T^o\colon (X^o,\mu^o)\to (X^o,\mu^o)$ as rank-one systems defined with the $(C,F)$-construction associated to the new parameters $(k_{2n})_{n\geq 0}$, $(\sigma_{2n})_{n\geq 0}$, $(q_{2n})_{n\geq 0}$ and $(k_{2n+1})_{n\geq 0}$, $(\sigma_{2n+1})_{n\geq 0}$, $(q_{2n+1})_{n\geq 0}$ respectively, where $X^e$ and $X^o$ are defined as direct limits as explained in Section~\ref{sec:PrelR1}, with all quantities indexed over even (resp.~odd) integers. To be more precise, $X^e$ is the direct limit of $(X_{2n}\coloneq F_{2n}\times C_{2n}\times C_{2n+2}\times\dots)_{n\geq 0}$ with connecting maps
\[(f_{2n},c_{2n},c_{2n+2},\dots)\in X_{2n}\mapsto (f_{2n}+c_{2n},c_{2n+2},\dots)\in X_{2n+2},\]
and we denote by $\iota^{(2n)}\colon X_{2n}\to X^e$ the canonical embedding. We define the set $X^o$ using odd indices in the same fashion. For every $n\geq 0$, we also denote by $\pi_n\colon \iota^{(n)}(X_n)\to F_n$ the canonical projection, defined by $\pi_n(\iota^{(n)}(f_n,c_n,c_{n+2},\dots))=f_n$. Note that we have $\pi_n=p_n\circ\pi_{n+2}$ on $\iota^{(n)}(X_n)$.

Since $(\sigma_{n-1}/k_{n+1})_n$ is summable, a first application of the Borel-Cantelli lemma implies that for almost every $x\in X^e$, we have $x\in\iota^{(2n)}(X_{2n})$ for large enough integers $n$, so that the quantity $\pi_{2n}(x)$ is well-defined for large values of $n$. The same holds for almost every $y\in X^o$, replacing $2n$ by $2n+1$.

Assumption~\eqref{item:ProofTheorem4}, Lemmas~\ref{lem:projection2} and~\ref{lem:ThisIsAnOE}, and a second application of the Borel-Cantelli lemma imply that for almost every $x\in X^e$ (resp.~$y\in X^o$), there exists an integer $N^e(x)\geq 0$ (resp.~$N^o(y)\geq 0$) such that for every $j\geq N^e(x)$ (resp.~$j\geq N^o(y)$) we have
\begin{equation}\label{eq:ProofTheorem1}
    \varphi_{2j+1}(\pi_{2(j+1)}(x))=p_{2j+1}\circ\varphi_{2j+3}(\pi_{2(j+2)}(x))\text{ and }\lambda_{\varphi_{2j+3}}(\pi_{2(j+2)}(x))=\lambda_{\varphi_{2j+1}}(\pi_{2(j+1)}(x))
\end{equation}
\begin{equation}\label{eq:ProofTheorem2}
    \text{(resp.~} \varphi_{2j}(\pi_{2j+1}(x))=p_{2j}\circ\varphi_{2j+2}(\pi_{2j+3}(x))\text{ and }\lambda_{\varphi_{2j+2}}(\pi_{2j+3}(y))=\lambda_{\varphi_{2j}}(\pi_{2j+1}(y))\text{)}.
\end{equation}

We set $z_{2j+1}(x)=\varphi_{2j+1}(\pi_{2(j+1)}(x))$ for every $j\geq N^e(x)$ (resp.~$z_{2j}(y)=\varphi_{2j}(\pi_{2j+1}(y))$ for every $j\geq N^o(y)$). The first equality in~\eqref{eq:ProofTheorem1} implies that $z_{2j+3}(x)-z_{2j+1}(x)$ is of the form $k-p_{2j+1}(k)$ (with $k\coloneq \varphi_{2j+3}(\pi_{2(j+2)}(x))\in F_{2j+1}+C_{2j+1}$), so it lies in $C_{2j+1}$, and similarly $z_{2j+2}(y)-z_{2j}(y)\in C_{2j}$ using~\eqref{eq:ProofTheorem2}. From this, we deduce that the quantity
\[\iota^{(2j+1)}( z_{2j+1}(x), z_{2j+3}(x)-z_{2j+1}(x),z_{2j+5}(x)-z_{2j+3}(x),\dots )\]
\[\text{(resp.~}\iota^{(2j)}( z_{2j}(y), z_{2j+2}(y)-z_{2j}(y),z_{2j+4}(y)-z_{2j+2}(y),\dots )\text{)}\]
is well-defined and independent of $j\geq N^e(x)$ (resp.~$j\geq N^o(y)$), and we denote it by $\varphi^o(x)$ (resp.~$\varphi^e(y)$).

The maps $\varphi^o\colon X^o\to X^o$ and $\varphi^e\colon X^e\to X^e$ are inverses of each other. Indeed, for almost every $y\in X^e$ and for a large enough integer $j$, we have
\begin{align*}
    \pi_{2j+1}(\varphi^o\circ\varphi^e(y))&=z_{2j+1}(\varphi^e(y))\\
    &=\varphi_{2j+1}(\pi_{2(j+1)}(\varphi^e(y)))\\
    &=\varphi_{2j+1}(z_{2(j+1)}(y))\\
    &=\varphi_{2j+1}\circ\varphi_{2(j+1)}\circ\pi_{2j+3}(y)\\
    &=p_{2j+1}\circ\underbrace{\pi_{2j+3}(y)}_{\in F_{2j+1}+C_{2j+1}}\\
    &=\pi_{2j+1}(y),
\end{align*}
which implies that $\varphi^o\circ\varphi^e=\id_{X^e}$ almost everywhere, and similarly $\varphi^e\circ\varphi^o=\id_{X^o}$.

Moreover, these maps preserve the probability measures. By symmetry, it is enough to check that $\mu^e((\varphi^o)^{-1}(A\times C_{2n+1}\times C_{2n+3}\times\dots))\leq\mu^o(A\times C_{2n+1}\times C_{2n+3}\times\dots)$ for every $n\geq 0$ and $A\subset F_{2n+1}$. Denoting $u_{2j}$ the measure of each level of the $j$-th tower for $T^o$, namely $u_{2j}\coloneq\mu^{e}(\{f_{2j}\}\times C_{2j}\times C_{2j+2}\times\dots)$ for any $f_{2j}\in F_{2j}$, and similarly for $u_{2j+1}$, we have
\begin{align*}
    &\mu^e((\varphi^o)^{-1}(A\times C_{2n+1}\times C_{2n+3}\times\dots))\\
    =&\ \lim_{j\to +\infty}{\mu^e(\{N^o\leq j\}\cap (\varphi^o)^{-1}(A\times C_{2n+1}\times C_{2n+3}\times\dots))}\\
    \leq &\ \limsup_{j\to +\infty}{\mu^e(z_{2j+1}^{-1}(A+ C_{2n+1}+\dots + C_{2j-1}))}\\
    \leq &\ \limsup_{j\to +\infty}{\mu^e(\{x\in \iota^{(2(j+1))}(X^{2(j+1)})\mid\pi_{2(j+1)}(x)\in\varphi_{2j+1}^{-1}(A+ C_{2n+1}+\dots +  C_{2j-1})\})}\\
    \leq &\ \limsup_{j\to +\infty}{u_{2j+2}}\cdot \left |\varphi_{2j+1}^{-1}(A+ C_{2n+1}+\dots + C_{2j-1})\right |\\
    \leq &\ \limsup_{j\to +\infty}{u_{2j+2}\cdot (d_{2j+2}k_{2j}\cdot\left |A\times C_{2n+1}\times\dots\times C_{2j-1}\right |+r_{2j+2})}\\
    \leq &\ \limsup_{j\to +\infty}{\bigg ( \frac{u_{2j+2}d_{2j+2}k_{2j}}{u_{2j+1}}\mu^o\bigg(A\times C_{2n+1}\times\dots\times C_{2j-1}\times \prod_{m\geq j}{F_{2m+1}}\bigg)+u_{2j+2}k_{2j}k_{2j+1}\bigg)}\\
    \leq &\ \mu^o (A\times C_{2n+1}\times C_{2n+3}\times\dots),
\end{align*}
where the last line uses the fact that $u_n\sim 1/k_n$.

It remains to prove that $\varphi^o$ and $\varphi^e$ are orbit equivalences between $T^e$ and $T^o$, with $\omega$-integrable cocycles. Since we have $\pi_{2j}(T^ex)=\pi_{2j}(x)+1$ for almost every $x\in X^e$ and for a large enough integer $j$, the second equality in~\eqref{eq:ProofTheorem1} exactly means that
\[z_{2j+3}(T^ex)-z_{2j+1}(T^ex)=z_{2j+3}(x)-z_{2j+1}(x).\]
This implies that for a large enough integer $j$, we have
\begin{align*}
    \varphi^o(T^ex)&=\iota^{(2j+1)}( z_{2j+1}(T^ex), z_{2j+3}(T^ex)-z_{2j+1}(T^ex),z_{2j+5}(T^ex)-z_{2j+3}(T^ex),\dots )\\
    &=\iota^{(2j+1)}( z_{2j+1}(T^ex), z_{2j+3}(x)-z_{2j+1}(x),z_{2j+5}(x)-z_{2j+3}(x),\dots )\\
    &=(T^o)^{c(x)}(\varphi^o(x))
\end{align*}
with $c(x)=z_{2j+1}(T^ex)-z_{2j+1}(x)$. We similarly get an equality of the form $\varphi^e(T^oy)=(T^e)^{c'(y)}\varphi^e(y)$, so the maps $\varphi^o$ and $\varphi^e$ are orbit equivalences between $T^e$ and $T^o$ with cocycles $c\colon X^e\to\Z$ and $c'\colon X^o\to\Z$ defined as above. We now prove that the cocycle $c\colon X^e\to\Z$ is $\omega$-integrable (the proof will be the same for the other cocycle $c'\colon X^o\to\Z$). Given $x\in X^e$, in the appropriate conull subset of $X^e$, we have $c(x)=z_{2j+1}(T^ex)-z_{2j+1}(x)$ for every large enough integer $j$, namely $c(x)=\lambda_{\varphi_{2j+1}}(p_{2(j+1)}(x))$. This means that the sequence $(c_{2j+1}\colon X^e\to\Z)_{j\geq 0}$ of maps defined by
\[
    \forall x\in X^e,\ c_{2j+1}(x)=\begin{cases}
                    \lambda_{\varphi_{2j+1}}(p_{2(j+1)}(x))&\text{if }x,T^ex\in \iota^{2(j+1)}(X_{2(j+1)})\\
                    0&\text{ otherwise}
                \end{cases}
\]
converges pointwise to $c\colon X^e\to\Z$. Fatou's lemma thus yields
\[\int_{X^e}{\omega(|c(x)|)\mathrm{d}\mu(x)}\leq \liminf_{j\to +\infty}{\int_{X^e}{\omega(|c_{2j+1}(x)|)\mathrm{d}\mu(x)}}=\liminf_{j\to +\infty}{u_{2(j+1)}\sum_{x\in [k_{2(j+1)}-1]}{\omega(|\lambda_{\varphi_{2j+1}}(x)|)}}.\]
We finally use Corollary~\ref{cor:CocyclePhiSeries}, assumption~\eqref{item:ProofTheorem3}, and the fact that $u_{2(j+1)}\sim \frac{1}{k_{2(j+1)}}$ to finally get that $\int_{X^e}{\omega(|c(x)|)\mathrm{d}\mu(x)}$ is finite. This concludes the proof.
\end{proof}

\section{Optimality of Belinskaya's theorem}\label{sec:Belinskaya}

We are now able to complete Carderi, Joseph, Le Maître, and Tessera's main result in~\cite{carderiBelinskayaTheoremOptimal2023}.

\begin{proof}[Proof of Theorem~\ref{TheoremB}]
By assumption, $S$ factors onto some odometer $S_0\in\aut$, i.e. there exists a measure-preserving map $\pi\colon X\to X$ such that $\pi\circ S= S_0\circ\pi$ almost everywhere. The point spectrum of $S$ being at most countable, there exists an irrational number $\theta$ such that $e^{2i\pi\theta}$ is not an eigenvalue of $S$. By Theorem~\ref{TheoremA}, $S_0$ and the irrational rotation $R_{\theta}$ of angle $\theta$ are $\omega$-integrably orbit equivalent. This means that $R_{\theta}$ is isomorphic to some $T_0\in\aut$ which has the same orbits as $S_0$ and such that the corresponding orbit cocycles $c,c'\colon X\to\Z$ of $\id_X$, defined by
\[S_0x=T_0^{c(x)}x\text{ and }T_0x=S_0^{c'(x)}x,\]
are $\omega$-integrable.

Let $T\in\aut$ be defined by $Tx=S^{c'(\pi(x))}x$. $T$ factors onto $T_0$, via the same factor map $\pi\colon X\to X$. Moreover, $T$ and $S$ have the same orbits, with cocycles $c\circ\pi$ and $c'\circ\pi$. Indeed, one inclusion between orbits is obvious by definition of $T$ and the other can be deduced from the identity $1=\sum_{i=0}^{c(\cdot)-1}{c'(T_0^i(\cdot))}$. This implies that $S$ and $T$ are $\omega$-integrable. Finally, since $T$ factors onto $R_{\theta}$, its point spectrum contains $e^{2i\pi\theta}$, so $S$ and $T$ have different point spectrums and are not flip-conjugate.
\end{proof}

\printbibliography

\end{document}